\documentclass[11pt]{amsart}

\usepackage{amssymb,amsmath,amsfonts,amsthm,mathrsfs}

\usepackage{color}
\usepackage[dvipsnames]{xcolor}
\usepackage{graphicx}
\usepackage{hyperref}
\usepackage{multicol}

\newcommand{\eps}{\varepsilon}
\newcommand{\R}{\mathbb{R}}

\usepackage{listings}
\definecolor{codegreen}{rgb}{0,0.6,0}
\definecolor{codegray}{rgb}{0.5,0.5,0.5}
\definecolor{codepurple}{rgb}{0.58,0,0.82}
\definecolor{backcolour}{rgb}{0.95,0.95,0.92}

\lstdefinestyle{mystyle}{
	backgroundcolor=\color{backcolour},   
	commentstyle=\color{codegreen},
	keywordstyle=\color{magenta},
	numberstyle=\tiny\color{codegray},
	stringstyle=\color{codepurple},
	basicstyle=\ttfamily\footnotesize,
	breakatwhitespace=false,         
	breaklines=true,                 
	captionpos=b,                    
	keepspaces=true,                 
	numbers=left,                    
	numbersep=5pt,                  
	showspaces=false,                
	showstringspaces=false,
	showtabs=false,                  
	tabsize=2
}

\theoremstyle{plain}
\newtheorem{thm}{Theorem}[section]
\newtheorem{prop}[thm]{Proposition}
\newtheorem{cor}[thm]{Corollary}
\newtheorem{lemma}[thm]{Lemma}

\theoremstyle{definition}
\newtheorem{defin}[thm]{Definition}
\newtheorem{ex}[thm]{Example}
\newtheorem{rmk}[thm]{Remark}
\newtheorem{conj}[thm]{Conjecture}

\usepackage{textcomp}
\usepackage{float}

\usepackage{graphicx}

\graphicspath{ {./images/} }

\calclayout 

\newcommand{\diag}{\operatorname{diag}}

\begin{document}

\title[Curvature sharpness and flow in weighted graphs -- Theory]{Bakry-\'Emery curvature sharpness and curvature flow in finite weighted graphs. I. Theory}
	
    \author[Cushing]{David Cushing}
    \address{School of Mathematics, Statistics and Physics, Newcastle University, Newcastle upon Tyne, Great Britain}
    \email{David.Cushing1024@gmail.com}
    \author[Kamtue]{Supanat Kamtue}
    \address{Yau Mathematical Sciences Center, Tsinghua University, Beijing, China}
    \email{skamtue@tsinghua.edu.cn}
    \author[Liu]{Shiping Liu}
    \address{School of Mathematical Sciences and CAS Wu Wen-Tsun Key Laboratory of Mathematics, University of Science and Technology of China, Hefei, China}
    \email{spliu@ustc.edu.cn}
    \author[M\"unch]{Florentin M\"unch}
    \address{Max Planck Institute for Mathematics in the Sciences, Leipzig, Germany} 
    \email{muench@mis.mpg.de}
    \author[Peyerimhoff]{Norbert Peyerimhoff}
    \address{Department of Mathematical Sciences, Durham University, Durham, Great Britain}
    \email{norbert.peyerimhoff@durham.ac.uk}
    \author[Snodgrass]{Ben Snodgrass}
    \address{Department of Mathematical Sciences, Durham University, DH1 3LE, Great Britain}
    \email{hugo.b.snodgrass@durham.ac.uk}
	\date{\today}
	
	\begin{abstract}
		In this sequence of two papers, we introduce a 
		curvature flow on (mixed) weighted graphs which is based on the Bakry-\'Emery calculus. The flow is described via a time-continuous evolution through the weighting schemes. By adapting this flow to preserve the Markovian property, its limits turn out to be curvature sharp. 
		Our aim is to present the flow in the most general case of not necessarily reversible random walks allowing laziness, including vanishing transition probabilities along some edges (``degenerate'' edges).
		This approach requires to extend all concepts (in particular, the Bakry-\'Emery curvature related notions) to this general case and it
		leads to a distinction between the underlying topology (a mixed combinatorial graph) and the weighting scheme (given by transition rates). We present various results about curvature sharp vertices and weighted graphs as well as some fundamental properties of this new curvature flow. 
		
		This paper is accompanied by a second paper discussing the curvature flow implementation in Python for practical use. In this second paper we present examples and exhibit further properties of the flow. 
	\end{abstract}
	
	\maketitle
	
	\tableofcontents
	
	\section{Introduction}
	
	This paper is based on a Ricci-type curvature notion for finite weighted graphs due to Bakry and \'Emery. We will introduce a continuous time curvature flow on its weights, and of special importance will be the concept of curvature sharpness. The flow aims to be convergent to curvature sharp weighting schemes. The paper presents various curvature sharpness results and some fundamental flow properties.
	
	\medskip
	
	Let us discuss the general setup. A \emph{weighted graph} is a simple mixed combinatorial graph $G=(V,E)$ (with vertex set $V$ and set $E = E^1 \cup E^2$ of one-sided and two-sided edges) together with a set of not necessarily symmetric weights $p_{xy} \ge 0$ (transition rates) which are only non-zero if $x=y$ or if there is a (one- or two-sided) edge from $x$ to $y$. Note however, that even if $p_{xx} > 0$, we do not have a loop at the vertex $x$, since we require our underlying graph to be simple (no loops, no multiple edges). The graph induces a generally non-symmetric combinatorial distance function $d_G: V \times V \to \mathbb{N} \cup \{0,\infty\}$, where $d_G(x,y)$ is the length of the shortest \emph{directed} path from $x$ to $y$. All our graphs are usually assumed to be finite and connected, but many results hold also true for locally finite infinite graphs. 
	
	By an  enumeration $V = \{ v_0,v_1,\dots,v_{n-1} \}$ of the vertices, 
	a weighted graph can be represented by its not necessarily symmetric adjacency matrix $A_G$ (describing the underlying topology) and a matrix $P$ of size $n$ whose entries $p_{ij}$ are the weights $p_{v_i v_j}$. We refer to the matrix $P$ as the \emph{weighting scheme} of the weighted graph $(G,P)$. The \emph{Bakry-\'Emery curvature} of a vertex $x \in V$, denoted by $K_N(x) = K_{P,N}(x)$, depends on a dimension parameter $N \in (0,\infty]$ and is based on the Laplacian
	\begin{equation} \label{eq:rwLap} 
		\Delta_P f(x) = \sum_{y \in V} p_{xy} (f(y)-f(x)). 
	\end{equation}
	The precise definition of this curvature notion is given in the next subsection. Usually we assume the weighting scheme to be \emph{Markovian}, that is, $P$ to be a stochastic matrix whose row entries add up to one. In the Markovian case we can interpret the values $p_{xy}$ for fixed $x \in V$ as transition probabilities of a random walk. The value $p_{xx}$ is then called 
	the \emph{laziness} of this random walk at the vertex $x$, and its value is irrelevant for the Laplacian \eqref{eq:rwLap}. Since the definition of Bakry-\'Emery curvature only involves the Laplacian, the laziness values $p_{xx}$ have also no influence on this curvature notion. In fact, the Bakry-\'Emery curvature $K_{P,N}(x)$ depends only on $P$ and not on the topology of the underlying graph $G$ (which can have edges even though the corresponding transition rates may be zero). Since the curvature of an \emph{isolated vertex} $x \in V$ (that is $p_{xy} = 0$ for all $y \neq x$) is a bit ambiguous, we define it to be zero at such a vertex for all dimensions. (Alternatively, there are also valid arguments to define it to be infinity.) The set of all Markovian weighting schemes $P$ corresponding to a mixed combinatorial graph $G$ is denoted by $\mathcal{M}_G$. We refer to a pair $(G,P)$ with $P \in \mathcal{M}_G$ as a \emph{Markovian weighted graph}. In this case, the matrix $P$ describes a time-homogeneous, discrete time Markov chain.
	
	In a weighted graph $(G,P)$, we distinguish between degenerate and non-degenerate edges and vertices. A
	one-sided edge $(x,y) \in E^1$ is called 
	\emph{degenerate} if $p_{xy} = 0$, and a two-sided edge $\{x,y\} \in E^2$ is called \emph{degenerate} if
	at least one of the transition probabilities $p_{xy}, p_{yx}$ is zero. Similarly, a vertex $x \in V$ is called \emph{degenerate}, if at least one of the transition probabilities $p_{xy}$ corresponding to one- and two-sided edges emanating from $x$ is zero. A weighted graph is called \emph{degenerate} if it has at least one degenerate edge or, equivalently, if it has at least one degenerate vertex. An example of a non-degenerate weighted graph $(G,P)$ is the (non-lazy) simple random walk $p_{xy} = \frac{1}{d_x}$ for all pairs $x \sim y$ of adjacent vertices, where $d_x$ denotes the combinatorial vertex degree of $x \in V$ (that is the cardinality of its outgoing one- and two-sided edges). 
	
	\medskip
	
	Since Bakry-\'Emery curvature and curvature sharpness are behind all our investigations, we start our paper with a brief introduction into these notions.
	
	\subsection{Bakry-\'Emery curvature and curvature sharpness} Bakry-\'Emery curvature can be defined on the states of a Markov chain given by the stochastic matrix $P$. Since we are in the setting of Markovian weighted graphs $(G,P)$, the (Markov chain) states correspond to the vertices of $G=(V,E)$. Bakry-\'Emery curvature is motivated by a fundamental identity in Riemannian geometry, called Bochner’s formula. The definition involves two ``carr\'e du champ operators" $\Gamma$ and $\Gamma_2$. More precisely, these operators are defined as
	\begin{eqnarray*}
		2 \Gamma(f,g) &=& \Delta(fg) - f \Delta g - g \Delta f, \\
		2\Gamma_2(f,g) &=& \Delta \Gamma(f,g) - \Gamma(f,\Delta g) - \Gamma(g, \Delta f), 
	\end{eqnarray*}
	where $\Delta=\Delta_P$ is the random walk Laplacian, acting on function $f: V \to \mathbb{R}$ and introduced in \eqref{eq:rwLap}. 
	\emph{Bakry-\'Emery} curvature is now defined as follows. \emph{Bakry-\'Emery curvature} was introduced for the smooth setting in \cite{BE84}. The curvature was then reintroduced several times in the setting of graphs, see \cite{Elworthy91,Schmuck,LinYau2010}. For further research about Bakry-\'Emery curvature on finite graphs, see \cite{Bauer17curvature, CLY14, CKKLP21, CKLLS19, CKPW20, FL22WarpedProd, FS18,HL19,KMY21,LMP18,LMP17rigidity,LMPR19,LP18,Ma13Bochner,Man15LogHarnark,Munch18Perpetual,Munch19LiYau,MR20,PESTGT16,Robertson19Harnack,Salez21sparse,Salez21cutoff,ShiYu20}.
	
	\begin{defin}[Bakry-\'Emery curvature]
		The \emph{Bakry-\'Emery curvature} of a vertex $x \in V$ for a fixed dimension $N \in (0,\infty]$ is the supremum of all values $K \in \mathbb{R}$, satisfying the \emph{curvature-dimension inequality}
		\begin{equation} \label{eq:cd-ineq} 
			\Gamma_2(f)(x) \ge \frac{1}{N} (\Delta f(x))^2 + K\, \Gamma(f)(x) 
		\end{equation}
		for all functions $f: V \to \mathbb{R}$. We use the simplified notation $\Gamma(f) = \Gamma(f,f)$ and $\Gamma_2(f) = \Gamma_2(f,f)$.
		We denote the curvature at $x \in V$ by $K_N(x) = K_{P,N}(x)$. 
	\end{defin} 
	
	In the case of a Markovian weighted graph $(G,P)$ and $x \in V$, we have
	$$ -1 \le K_N(x) \le 2 $$
	for all dimensions $N \ge 2$. 
	Even more precise lower and upper curvature bounds are given in Theorem \ref{thm:curv-bound}.  Any function $f: V \to \mathbb{R}$ with $\Gamma(f)(x) \neq 0$ gives rise to an upper curvature bound
	$K_N(x) \le K_N^f(x)$ with
	\begin{equation} \label{eq:K_f} 
		K_N^f(x) := \frac{1}{\Gamma(f)(x)} \left( \Gamma_2(f)(x) - \frac{1}{N} (\Delta f(x))^2 \right).
	\end{equation}
	In fact, $K_N(x)$ is the infimum of all values $K_N^f(x)$ of such functions $f$, that is
	\begin{equation} \label{eq:var-princ-curv} 
		K_N(x) = \inf_{f:\, \Gamma(f)(x) \neq 0} K_N^f(x). 
	\end{equation}
	The upper curvature bound for the particular function $f(\cdot) = d_G(x,\cdot)$ leads to the notion of curvature sharpness which will be of central importance in this paper. The idea to use the combinatorial distance function for an upper curvature bound goes back to \cite[Proof of Theorem 1.2]{KKRT16}.
	
	\begin{defin}[Curvature sharpness] \label{def:curv_sharp}
		Let $(G,P)$ be a Markovian weighted graph.
		A vertex $x \in V$ is called \emph{$N$-curvature sharp} (for a fixed dimension parameter $N$) if
		$$ K_N(x) = K_N^{d_G(x,\cdot)}(x), $$
		that is, if the infimum in \eqref{eq:var-princ-curv} is assumed for the combinatorial distance function $f(\cdot) = d_G(x,\cdot)$. Moreover, a vertex $x$ is
		called \emph{curvature sharp} if it is curvature sharp for some dimension $N \in(0,\infty]$, and $(G,P)$ is called \emph{curvature sharp} if each vertex of $G$ is curvature sharp.
	\end{defin}
	
	
	Note that, while Bakry-\'Emery curvature at a vertex $x \in V$ of a weighted graph $(G,P)$ depends only on the weighting scheme $P$, curvature sharpness is based on the combinatorial distance function $d_G(x,\cdot)$ of the graph $G$, which means that this notion depends also on the underlying topology. 
	Examples of \emph{non-degenerate} curvature sharp weighted graphs are simple random walks without laziness on complete graphs $K_n$ (with constant curvature $K_\infty(x) = \frac{1}{2} + \frac{3}{2(n-1)}$) or simple random walks without laziness on triangle-free $d$-regular graphs (with $K_\infty(x) \le \frac{2}{d}$). 
	
	\medskip
	
	Curvature sharpness was originally introduced in \cite{CLP-20} for the non-normalized Laplacian on combinatorial graphs and in \cite{CKLP-21} for general weighted graphs. The curvature sharpness definitions in these papers differ from the one given here. We will see in due course that all these definitions are consistent with each other. 
	
	\medskip
	
	Following ideas in \cite{CKLP-21} (see also \cite{Sic21} for the unweighted case), the curvature $K_\infty(x)$ of a \emph{non-degenerate} vertex $x \in V$ can also be expressed as the smallest eigenvalue of a particular symmetric matrix $A_\infty(x) = A_{P,\infty}(x)$ of size $d_x$, the number of outgoing edges from $x$. The matrix $A_\infty(x)$ is called the \emph{curvature matrix} at $x \in V$, and we view it as a discrete version of the Ricci curvature tensor acting as a quadratic form on the tangent space $T_xM$ in the smooth setting of a Riemannian manifold $(M,g)$. This viewpoint will be the guiding principle in the next subsection. In order not to overburden this introduction, we refer readers to \cite[(1.2)]{CKLP-21} for more details about this matrix $A_\infty(x)$ and its definition.
	
	\subsection{A curvature flow}
	The motivation for the work in this paper was to introduce a (Bakry-\'Emery) curvature flow on finite weighted graphs, similar in spirit to the Ricci curvature flow of a Riemannian manifold. 
	
	\medskip
	
	Let us first explain the curvature flow in the special case of a \emph{non-degenerate} weighted graph $(G,P)$. As mentioned before, the Bakry-\'Emery curvature of a vertex $x \in V$ can be expressed in this case as a minimal eigenvalue, namely,
	$$ K_N(x) = \lambda_{\rm{min}}(A_N(x)) $$
	with
	\begin{equation} \label{eq:ANAinf}
		A_N(x) := A_\infty(x) - \frac{2}{N}
		{\bf{v}}_0(x) {\bf{v}}_0(x)^\top, 
	\end{equation}
	where $A_\infty(x) = A_{P,\infty}(x)$ is a special symmetric matrix (the curvature matrix) of size $m = d_x$ corresponding to an enumeration $y_1,\dots,y_m$ of the vertices adjacent to $x$ 
	and 
	$$ {\bf{v}}_0(x) = (\sqrt{p_{xy_1}},\dots,\sqrt{p_{xy_m}})^\top. $$
    As mentioned before, the matrix $A_\infty(x)$ is defined in \cite[(1.2)]{CKLP-21} and is related to another matrix $Q(x)$ which is defined as a Schur complement of a matrix $\Gamma_2(x)$ related to the $\Gamma_2$-operator earlier (see formula \eqref{eq:Qx} below).
	The variational eigenvalue characterisation via the Rayleigh quotient yields the estimate
	\begin{equation} \label{eq:KN-Rayleigh} 
		K_N(x) \le \frac{{\bf{v}}_0(x)^\top A_N(x){\bf{v}}_0(x)}{{\bf{v}}_0(x)^\top {\bf{v}}_0(x)} 
	\end{equation}
	where the right hand side turns out to agree with the upper curvature bound $K_N^{d_G(x,\cdot)}(x)$, defined earlier in \eqref{eq:K_f} (see Proposition \ref{prop:coinc-curv-bounds}). Henceforth, we will denote the right hand side of \eqref{eq:KN-Rayleigh} by $K_N^0(x)$, that is 
	\begin{equation} \label{eq:KN0first} 
		K_N^0(x) = K_{P,N}^0(x) := \frac{{\bf{v}}_0(x)^\top A_N(x){\bf{v}}_0(x)}{{\bf{v}}_0(x)^\top {\bf{v}}_0(x)}. 
	\end{equation}
	These considerations provide an alternative equivalent description of curvature sharpness in the non-degenerate case (see also \cite[Proposition 1.7(i)]{CKLP-21}): $x \in V$ is curvature sharp if and only if
	\begin{equation} \label{eq:curv-sharp-ainf} 
		A_\infty(x) {\bf{v}}_0(x) = K_\infty^0(x) {\bf{v}}_0(x). 
	\end{equation}
	
	\medskip
	
	Our first Ansatz for the curvature flow was the system of ordinary differential equations
	\begin{equation} \label{eq:curv-flow0} 
		{\bf{v}}_0'(x,t) = - A_{P(t),\infty}(x) {\bf{v}}_0(x,t), 
	\end{equation}
	with one such equation for every vertex $x \in V$, and with initial condition $P(0) = P_0$. Moreover, we added the conditions
	\begin{equation} \label{eq:curv-flow0laz}
	p_{xx}'(t)= 0 \quad \text{for all $x \in V$,}
	\end{equation}
	which means that the laziness values do not change in time.
	Note that $P(t)$ and the system of vectors ${\bf{v}}_0(x,t)$ for all $x \in V$ determine each other mutually. Left multiplication of \eqref{eq:curv-flow0} with $2{\bf{v}}_0(x,t)^\top$ leads to the equations
	\[
	 \frac{d}{dt} \langle {\bf{v}}_0(x,t),{\bf{v}}_0(x,t) \rangle = - 2 {\bf{v}}_0(x,t)^\top A_{P(t),\infty}(x) {\bf{v}}_0(x,t),
	\]
	for all $x \in V$, which resemble the equation
	\[
	 \frac{d}{dt} g_t = - 2 {\rm{Ric}}_{g_t}
	\]
	of the Ricci flow in the setting of Riemannian manifolds $(M,g_t)$. 
	
	\medskip
	
	A problem with the flow \eqref{eq:curv-flow0} is that it does not preserve the Markovian property. 
	For that reason, we modified the differential equations in \eqref{eq:curv-flow0} as follows:
	\begin{equation} \label{eq:curv-flow} 
		{\bf{v}}_0'(x,t) = - A_{P(t),\infty}(x) {\bf{v}}_0(x,t) + K_{P(t),\infty}^0(x) {\bf{v}}_0(x,t).
	\end{equation}
	with the additional term $K_{P,\infty}^0(x)$ defined above in \eqref{eq:KN0first}. It turns out that this modification preserves the Markovian property. It is in some sense similar to the idea of a normalized Ricci curvature flow in the setting of Riemannian manifolds, which has the property to be volume preserving. 
	
	\medskip
	
	Recall that we restricted our above considerations to the case of \emph{non-degenerate} weighting schemes. While the curvature flow in \eqref{eq:curv-flow} preserves non-degeneracy for finite times, it is desirable to extend it to the more general setting where we allow degeneracy. Moreover, our experiments with this flow indicate that the solution $P(t)$ of \eqref{eq:curv-flow} may be always convergent as $t \to \infty$, and that the limiting weighting scheme
	$P^\infty = \lim_{t \to \infty} P(t)$ is usually no longer non-degenerate, even if we start with a non-denegerate $P_0 \in \mathcal{M}_G$. Fortunately, it is possible to define the curvature flow without the non-degeneracy restriction. However, this generalization requires a description
	of the curvature flow which no longer involves the curvature matrix $A_\infty(x)$ but another matrix $Q(x)$, as will be briefly explained in Subsection 
	\ref{subsec:results} and in more detail in Subsections \ref{subsec:schur-reform} and
	\ref{subsec:der-curv-flow}.
	
	\medskip
	
	The equations \eqref{eq:curv-flow} imply that that stationary solutions $(G,P_0)$ of the curvature flow are precisely those weighting schemes $P_0 \in \mathcal{M}_G$ for which $(G,P_0)$ is curvature sharp (using the description given in  \eqref{eq:curv-sharp-ainf}). 
	Moreover, these equations imply that, starting with $P_0 \in \mathcal{M}_G$ and in case of convergence $P^\infty = \lim_{t \to \infty} P(t)$, the vectors ${\bf{v}}_0^\infty(x)$ of the limiting weighting scheme $P^\infty \in \mathcal{M}_G$ satisfy the equation \eqref{eq:curv-sharp-ainf}. In other word, any limit of the curvature flow as $t \to \infty$ represents a curvature sharp weighted graph. This fact emphasises the relevance of this curvature related concept, and a substantial part of this paper will be concerned with the study of curvature sharpness.  
	
	\subsection{Structure of the paper and results}
	\label{subsec:results}
	
    This paper has three parts. In the first very substantial part (Sections \ref{sec:curvandSchur} and \ref{sec:analgeomcurvsharpvert}) we introduce all relevant notions associated to a very general setting of Markovian weighted graphs $(G,P)$ and derive fundamental properties of curvature sharp vertices. The second part (Section \ref{sec:cursharpgraphs}) is concerned with relations between a combinatorial graph $G=(V,E)$ and associated Markovian weighting schemes $P$ such that all vertices in $V$ are curvature sharp. In the third part (Sections \ref{sec:curvcont} and \ref{sec:prop-curv-flow}) we introduce the curvature flow as a particular system of differential equations. Given an initial Markovian weighted graph $(G,P_0)$, this flow is represented by a smooth matrix valued function $P(t)_{t \in [0,\infty)}$ in the space of stochastic matrices with $P(0) = P_0$. Our experiments indicate that the curvature flow may always converge to a well-defined limit $P^\infty = \lim_{t \to \infty} P(t)$. It is also often the case that many vertices of the limit $(G,P^\infty)$ are degenerate even if the flow started at a non-degenerate weighted Markov chain $(G,P_0)$. For that reason, we discuss in Section \ref{sec:curvcont} semicontinuity (Theorem \ref{thm:curvsemicont}) and jump phenomena (Examples \ref{ex:curvjump1} and \ref{ex:curvjump2}) of Bakry-\'Emery curvatures as functions of the underlying weigting schemes $P$. The last section of this paper is concerned with the proof of some fundamental properties of this curvature flow. In a follow-up paper \cite{CKLMPS-22} we will discuss the implementation of this flow and various experimental results.
    
    \medskip
    
    Let us now turn to the discussion of the main results of this paper. Markovian weighted graphs $(G,P)$ have a (non-symmetric) distance function $d_G$. Since we allow degenerate vertices, there exists a natural mixed subgraph $(G_P,P)$ which is non-degenerate and which is obtained by dropping all edges which have vanishing transition rates (details are given in Definitions \ref{def:mixedgraph} and \ref{def:weightedgraph}). This subgraph
    has its own distance function $d_P$ with
    $d_P \ge d_G$. While Bakry-\'Emery curvature $K_N(x)$ at a vertex $x$ depends only on the weighting scheme $P$, curvature sharpness depends on the distance function (see Definition \ref{def:curv_sharp}), and it is therefore important to understand relations between curvature sharpness with respect to $d_G$ and with respect to $d_P$. It turns out that curvature sharp vertices in $(G,P)$ are also curvature sharp in $(G_P,P)$ but not vice-versa (see Proposition \ref{prop:curvsharpsubgraph}). Our main results about curvature sharp vertices in $(G,P)$ is listed in the following theorem, which is a short version of a more extensive collection of results presented in Theorem 
    \ref{thm:curvsharpeq}.
    
    \begin{thm}[Curvature sharpness at vertices] \label{thm:main-curv-sharp-vertex}
        Let $(G,P)$ be a Markovian weighted graph. Then the following statements about a vertex $x \in V$ are equivalent:
        \begin{itemize}
           \item[(1)] $x$ is curvature sharp in $(G,P)$. 
           \item[(2)] $x$ is $2$-curvature sharp
           in $(G,P)$.
           \item[(3)] We have
           $$ \label{eq:curv-sharp-with-Q}
           Q(x){\bf{1}}_m = \frac{1}{2}K_\infty^{d_G(x,\cdot)}(x) {\bf{p}}_x, 
           $$
        \end{itemize}
        with $Q(x)$ defined in \eqref{eq:Qx} below.
    \end{thm}

    It is important to clarify the notions used in statement (3) of the theorem: ${\bf{1}}_m$ is the all-one column vector of size $m$, with $m$ the cardinality of the $1$-sphere 
    $$ S_1(x) = \{ y \in V: d_G(x,y) = 1 \} = \{y_1,\dots,y_m\}, $$ 
    and ${\bf{p}}_x = (p_{xy_1},\dots,p_{xy_m})^\top$.
    The constant $K_\infty^{d_G(x,\cdot)}(x)$ is 
    a special case of the upper curvature bound $K_N^f(x)$, defined earlier as
    $$ K_N^f(x) = \frac{1}{\Gamma(f)(x)} \left( \Gamma_2(f)(x) - \frac{1}{N} (\Delta f(x))^2 \right). $$ 
    
    Let us now discuss the matrix $Q(x)$: We mentioned before that -- in the non-degenerate case -- Bakry-\'Emery curvature $K_N(x)$ agrees with the smallest eigenvalue of the curvature matrix $A_N(x)$. This symmetric curvature matrix is derived from the matrix $Q(x)$, which is obtained via a Schur complement construction associated to the $\Gamma_2$-operato (details can be found in Subsection \ref{subsec:schur-reform}). 
    This symmetric matrix $Q$ is uniquely determined by the relation ${\bf{v}}^\top Q(x) {\bf{v}} = \min \Gamma_2(f)$, where the minimum runs over all $f: V \to \mathbb{R}$ with $f(x) = 0$ and $f = {\bf{v}}$ on $S_1(x)$ (see Proposition \ref{prop:Qcharacterisation}). While $A_N(x)$ exists only in the case of a non-degenerate vertex $x$ (since its derivation requires to divide the $(i,j)$-th entries of $2 Q(x)$ by the transition rate expressions $\sqrt{p_{xy_i}p_{xy_j}}$, see \cite[(1.2)]{CKLP-21}), the matrix $Q(x)$ can also be defined for degenerate vertices, and the identity in (3) of Theorem \ref{thm:main-curv-sharp-vertex} can then be viewed as a generalisation of the curvature sharpness description \eqref{eq:curv-sharp-ainf} for non-degenerate weighted graphs to the case of weighted graphs with
    both degenerate and non-degenerate vertices.  
    
    Before we discuss other results, let us briefly reflect on the equivalences in Theorem \ref{thm:main-curv-sharp-vertex}: 
    Recall from Definition \ref{def:curv_sharp} that a vertex
    is called curvature sharp if it is curvature sharp for some dimension $N \in (0,\infty]$, that is, the Bakry-\'Emery curvature $K_N(x)$ agrees with the upper curvature bound $K_N^{d_G(x,\cdot)}(x)$. 
    Moreover, curvature sharpness for dimension $N$ implies curvature sharpness for any smaller dimension $N' \le N$ (see Proposition \ref{prop:curvsharpmon}). The equivalence ``$(1) \Leftrightarrow (2)$'' states that there is universal threshold, namely $N=2$, such that any curvature sharp vertex is automatically $2$-curvature sharp and, therefore, is curvature sharp for all dimensions in the range $(0,2]$. It is by no means obvious that such a universal threshold exists. We also show that this is a maximal threshold, that is, there exist curvature sharp vertices which are not curvature sharp for any dimension $N > 2$ (see Remark \ref{rmk:curv-sharp-threshold}). The equivalence ``$(1) \Leftrightarrow (3)$''
    provides a description of curvature sharpness via an explicit identity. This identity has various important consequences: It implies that curvature sharpness is fully determined by the transition rates of the one-ball
    $$ B_1(x) = \{y \in V: d_G(x,y) \le 1\} $$
    (see Theorem \ref{thm:curvsharpanal}). This is surprising since neither Bakry-\'Emery curvature nor $\infty$-curvature sharpness are already determined by the one-ball -- both require information about the transition rates of the $2$-ball. Another consequence of this identity becomes apparent when we generalize the curvature flow equations \eqref{eq:curv-flow0laz} and
    \eqref{eq:curv-flow} to be applicable to Markovian weighted graphs with both degenerate and non-degenerate vertices, using $Q$-matrix reformulations. These reformulations read as follows (see Subsection \ref{subsec:der-curv-flow} for details):
    
\begin{defin}[Curvature flow] \label{def:curvflow-Q}
    Let $(G,P_0)$ be a finite Markovian weighted graph. The associated curvature flow is given by the differential equations for all $x \in V$ and all $t \ge 0$:
    \begin{align*}
    p_{xx}'(t) &= 0, \\
    {\bf{p}}_x'(t) &= -4 Q_x(t){\bf{1}}_m + 2 K_{P(t),\infty}^{d_G(x,\cdot)}(x) {\bf{p}}_x(t), 
    \end{align*}
    with $m = |S_1(x)|$, $S_1(x) = \{y_1,\dots,y_m\}$,
    $$ {\bf{p}}_x(t) = (p_{xy_1}(t),\dots,p_{xy_m}(t)), $$
    and inital condition $P(0) = P_0$. 
\end{defin}
    
Note that our $Q$-matrices in Definition \ref{def:curvflow-Q} depend on both the vertex $x$ and the time parameter $t$, and we express here the vertex dependence of $Q$ by using $x$ as an index. It follows from the identity in (3) of Theorem \ref{thm:main-curv-sharp-vertex}
that any curvature sharp Markov chain $(G,P)$ is a stationary solution of the curvature flow.
	

\begin{rmk}
An alternative description of the curvature flow using the $\Gamma$-calculus is the following description 
$$ -\partial_t \Delta^{P(t)} f(x) = 4 \left( \Gamma_2^{P(t)}(d(x,\cdot),f)(x) - \frac{\Gamma_2^{P(t)}(d(x,\cdot))(x)}{(\Delta^{P(t)} d(x,\cdot))(x)} \Delta^{P(t)} f(x) \right) \quad \text{for all $f: V \to \mathbb{R}$}, $$
where $f$ is independent of time $t$ and $\Delta^{P(t)}$ and $\Gamma_2^{P(t)}$ are operators defined via the weighting scheme $P(t)$. Note that $\Delta^{P(t)}({\bf{1}}_y)(x) = p_{xy}(t)$, so that this equation determines the weighting scheme evolution $P(t)$ under the curvature flow. Obviously, we have $\partial_t \Delta^{P(t)}(d(x,\cdot))(x) = 0$, which means that the laziness value at $x$ is preserved. Note that, in general, $\Gamma_2(f,g)(x)$
depends on $f\vert_{B_2(x)}$. However, if $g = d(x,\cdot)$, then $\Gamma_2(f,g)(x)$ depends only on $f\vert_{B_1(x)}$.
\end{rmk}
 
Let us now present two other basic properties of our curvature flow, which are both proved in Subsection \ref{subsec:markovandcurvsharplimit-flow}. Our first result is that it is defined for all times $t \ge 0$ and that it preserves the Markovian property. 
	
\begin{thm}[Curvature flow preserves Markovian property] \label{thm:flow-pres-MP}
    Let $(G,P_0)$ be a Markovian weighted graph. Then the curvature flow $(G,P(t))_{t \ge 0}$ (given in Definition \ref{def:curvflow-Q}) associated to $(G,P_0)$ is well defined for all $t \ge 0$ and preserves the Markovian property. If $(G,P_0)$ is non-degenerate, then $(G,P(t))$ is also non-degenerate for all $t \ge 0$.
\end{thm}
	
The second flow result is about its limits -- it is another consequence of Theorem \ref{thm:main-curv-sharp-vertex} and the flow equations. Moreover, note that a flow limit often represents a degenerate weighted graph, even in the case when $(G,P_0)$ is non-degenerate. In fact, our experiments show that this is almost always the case.

\begin{prop}[Curvature sharpness of curvature flow limit]\label{prop:flow-lim-CS}
    Let $(G,P_0)$ be a Markovian weighted graph such that the curvature flow $(G,P(t))$ converges for $t \to \infty$ with $P^\infty = \lim_{t \to \infty} P(t)$. Then the weighted graph $(G,P^\infty)$ is curvature sharp. 
\end{prop}


As mentioned before, we do not know of any initial condition $(G,P_0)$ for which the curvature flow does not converge. This observation suggests the following conjecture.

\begin{conj}[Curvature flow convergence]
    The curvature flow converges for any initial condition $(G,P_0)$, that is, $(G,P(t))$ has a well defined limit 
    $P^\infty = \lim_{t \to \infty} P(t)$.
\end{conj}

In the remainder of this subsection, we only consider Markovian weighted graphs $(G,P)$ which do not have one-sided edges and all vertices have vanishing laziness. In this case, we say that $G=(V,E)$ is \textbf{unmixed} and that its weighting scheme $P$ is \textbf{without laziness}. 

For complete graphs $K_n$ with $n$ vertices, our observations support the following conjecture.

\begin{conj}[Curvature flow of complete graphs]     \label{conj:compl-gr}
    If $P_0$ is a non-degenerate Markovian weighting scheme without laziness on the complete graph $K_n$ with $n \ge 2$, then the curvature flow has a limit $P^\infty$ which is the simple random walk, that is, $p_{xy}^\infty = 1/(n-1)$ for all pairs $x,y \in V$ of different vertices.
\end{conj}

Since complete graphs have various degenerate curvature sharp weighting schemes we cannot drop the non-degeneracy assumption in the above conjecture.
Complete graphs are amongst the few examples where non-degenerate initial weighting schemes seem always to converge to a non-degenerate limit. If we could prove this observation, the following result would then confirm Conjecture \ref{conj:compl-gr} in the case $n=3$.
	
\begin{prop}[Curvature sharp weighting schemes for $K_3$] \label{prop:K3}
    Let $K_3$ be the complete graph with vertex set $\{ 0,1,2 \}$
    and Markovian weighting schemes $P = (p_{ij})$ without laziness represented by the vectors
    $$ \left(p_{01},p_{02},p_{10},p_{12},p_{20},p_{21}\right). $$
    There are only four such curvature sharp weighting schemes on $K_3$, namely
    $$ \left(\frac{1}{2},\frac{1}{2},\frac{1}{2},\frac{1}{2},\frac{1}{2},\frac{1}{2}\right), \left(0,1,\frac{1}{2},\frac{1}{2},1,0 \right), \left(\frac{1}{2},\frac{1}{2},0,1,0,1\right), \left(1,0,1,0,\frac{1}{2},\frac{1}{2}\right). $$
    Consequently, the simple random walk is the only non-degenerate curvature sharp Markovian weighting scheme without laziness on $K_3$.
\end{prop}
	
Proposition \ref{prop:K3} is proved in Subsection \ref{subsec:curvsharp_complete} with the help of symbolic Maple computations. 
It is conceivable that the only non-degenerate curvature sharp weighting scheme without laziness on any complete graph $K_n$ is the simple random walk, but we are currently not able to prove this for any natural number $n \ge 4$. 

\smallskip

For many interesting practical features of this curvature flow, we refer readers to our second paper \cite{CKLMPS-22}, which contains many examples and further observations and is accompanied by a Python program which can be used to run a numerical curvature flow. 

\smallskip
 
Let us now shift our attention to curvature sharp Markov weighted graphs. The following result shows that every unmixed graph $G$ admits a (generally degenerate) weakly connected curvature sharp Markovian weigthing scheme without laziness. It is proved in Subsection \ref{subsec:curv-sh-tr-fr-edge}.

\begin{thm}[Weakly connected curvature sharp weighting schemes] \label{thm:weakconncurvsharp}
Let $G=(V,E)$ be a simple finite connected unmixed graph and $G_0=(V_0,E_0)$ be a complete subgraph (clique) with $n = |V_0| \ge 2$. Then there exists a Markovian weighting scheme $P$ without laziness which is curvature sharp on $(G,P)$ such that its (mixed) induced subgraph $G_P = (V,E_P)$ is weakly connected and the restriction of $P$ to $G_0$ is a simple random walk.

\smallskip

The weighting scheme $P = (p_{xy})_{x,y} \in V$ has the following explicit description: 
\begin{itemize}
    \item[(i)] We have $p_{xy} = \frac{1}{n-1}$ for all $\{x,y\} \in E_0$ (that is, we have a simple random walk in $G_0$),
    \item[(ii)] for every $x \in V$ with $d_G(x,V_0)=1$, we have $p_{xy} = \frac{1}{k}$ for all vertices $y \in V_0$ adjacent to $x$, where $k$ is the number of neighbours of $x$ in $G_0$,
    \item[(iii)] for every $x \in V$ with $d_G(x,V_0)\ge 2$, 
    there exists a unique $y \in V$ with $d_G(y,V_0) = d_G(x,V_0)-1$ such that $p_{xy} = 1$.
\end{itemize}
\end{thm}
	
The importance of this theorem stems from the fact that it provides some support for our observation that any initial Markovian weighted graph is convergent under the curvature flow: recall that its limit must be curvature sharp and Theorem \ref{thm:weakconncurvsharp} confirms that curvature sharp weighting schemes exist for any initial combinatorial configuration.
	
The existence of a \emph{non-degenerate} curvature sharp Markovian weighting schemes $P$ is rare if the underlying combinatorial graph $G$ has a leaf. 

\begin{prop}[Curvature sharp weighted graphs with leaves] 
\label{prop:leaf-case} 
    Let $G = (V,E)$ be an unmixed simple finite connected graph and $y \in V$ be a leaf, that is, $y$ has vertex degree $1$. Then $G$ does not admit a non-degenerate curvature sharp Markovian weighting scheme without laziness unless $G$ is a star graph (that is, there is a vertex $x \in V$ with $V = \{x\} \cup S_1(x)$ and there are no edges between pairs of vertices in $S_1(x)$).
\end{prop}

Note that any random walk without laziness is a non-degenerate curvature sharp weighting scheme on a star graph (see Example \ref{ex:stargraph}). Proposition \ref{prop:leaf-case} as well as the next three results are proved in Subsection \ref{subsec:curv-sh-tr-free}. The first is a general result about non-degenerate triangle-free curvature sharp Markovian weighted graphs. This class comprises all bipartite graphs $G$. 

\begin{thm}[Curvature sharpness for triangle-free graphs] \label{thm:curv-sh-tr-free}
    Let $G=(V,E)$ an unmixed simple finite connected graph without triangles and $A_G$ be its adjacency matrix.
    Then the set of all non-degenerate curvature sharp Markovian weighting schemes without laziness is in $1-1$ correspondence to solutions ${\bf{c}} \in (0,1]^V$ of the equation
    $$ A_G\, {\bf{c}} = {\bf{1}}_M, $$
    where ${\bf{1}}_M$ is the all-one vector of size $M = |V|$. This correspondence is given by the relation $p_{xy} = c_y$ for all $x,y \in V$ with $x \sim y$. In particular, the graph $G$ cannot have a unique non-degenerate curvature sharp Markovian weighting scheme without laziness unless $A_G$ is invertible.

    \smallskip

    Moreover, the set of all non-degenerate curvature sharp Markovian weighting schemes without laziness on $G$ is a convex set.
\end{thm}
	
This theorem 
has the following two immediate consequences. 

\begin{cor}[Unique curvature sharpness for bipartite graphs] \label{cor:bip-uniq-curv-sharp} 
    Let $G = (V,E)$ be an unmixed simple finite connected bipartite graph. If $|V|$ is odd, then $G$ does not have a unique non-degenerate curvature sharp Markovian weighting scheme without laziness.
\end{cor}

We like to mention that bipartite graphs may not admit any non-degenerate curvature sharp weighting schemes, so this corollary is only a statement about their uniqueness and not about their existence. For example, any path of length $\ge 3$ does not admit non-degenerate curvature sharp Markovian weighting schemes without laziness because of Proposition \ref{prop:leaf-case}.

\begin{cor}[Unique curvature sharpness for hypercubes] \label{cor-hypcub}
    The simple random walk without laziness on the hypercube $Q^n = (K_2)^n$ is curvature sharp. It is the only non-degenerate curvature sharp Markovian weigthing scheme without laziness on $Q^n$ if and only if $n$ is odd.
\end{cor}
	
This finishes our list of results in this paper. There are various other results presented throughout the paper which are of their own interest like, e.g., the lower and upper curvature bounds in Theorem \ref{thm:curv-bound}. 
	
	
	
	\subsection{Other curvature flows on discrete spaces}
	To our knowledge, the curvature flow on discrete Markov chains in this paper is the first one which is based on Bakry-\'Emery curvature.  However, curvature flows were introduced for various other curvature notions for discrete spaces 
	like, e.g., networks, weighted graphs, simplicial complexes or discrete Markov chains.
	\begin{itemize}
	\item Combinatorial Ricci flows on surfaces
	were introduced in \cite{CL-03} in connection with a discretization of the Uniformization Theorem via circle packings.
    \item Weber et al. introduced a Forman-Ricci curvature flow in \cite{WSJ17FormanRicflow} and discussed applications for data mining, including denoising and clustering of experimental data, as well as extrapolation of network evolution. For theoretical background of Forman-Ricci curvature, see \cite{Forman,JM21}
	\item The question about a reasonable curvature flow for Ollivier-Ricci curvature was already mentioned as Problem N in Ollivier's problem list
	\cite{Ollivier,OllProbs}. Ni et al. introduced a slight modification of Ollivier's proposal in \cite{Ni19commDetect} with a view on community detection and Bai et al. \cite{Bai21OllRicflow} investigated existence and uniqueness properties of its solutions. 
	\item A discrete version of a super Ricci flow for weighted graphs with respect to entropic curvature was introduced in \cite{EM12,EK20superRicflow} with a discussion about its connections to the heat flow. 
	\item Balanced Forman curvature was introduced in \cite{TDGCDB-22} and can be viewed as a hybrid between Forman and Ollivier-Ricci curvature. Their stochastic discrete Ricci flow algorithm is based on the idea to remove edges with high curvature and to replace them by new edges around somewhere with low curvature. This algorithm is designed to avoid oversquashing of graph neural networks. Intuitively this means that it increases the Cheeger constant by getting rid of bottlenecks.
	\item The resistance curvature defined in \cite{DL-22} seems like a Forman curvature with respect to the (non-local) resistance metric. The authors introduce an associated Ricci flow and show that the finite path graph contracts under the flow as expected. 
    \end{itemize}
	
	
	\section{Bakry-\'Emery curvature and Schur complement}
	\label{sec:curvandSchur}
	
	The main results in this section are lower and upper curvature bounds given in Theorem \ref{thm:curv-bound}, a reformulation of Bakry-\'Emery curvature via the Schur complement and using the $Q$-matrix in Proposition \ref{prop:curvreform}, agreement of the upper curvature bounds $K_N^{d_P(x,\cdot)}(x)$ and $K_N^0(x)$ in the non-degenerate case, stated in Proposition \ref{prop:coinc-curv-bounds}, and relations between the upper curvature bounds $K_N^{d(x,\cdot)}(x)$ for different distance functions $d$ corresponding to subgraphs, given in Theorem \ref{thm:uppbds_subgraphs}.
	
	\subsection{Graph theoretical notions}
	
	In this subsection, we collect graph theoretical notions for combinatorial mixed graphs and for weighted graphs. Some of these notions were already mentioned in the Introduction.
	
	\begin{defin} \label{def:mixedgraph} A \textbf{mixed graph} $G=(V,E)$ has a vertex set $V$ and an edge set $E = E^1 \cup E^2$ comprising one- and two-sided edges. Any pair $x,y \in V$ of distinct vertices can either be non-adjacent or connected by a two-sided edge $\{x,y\} \in E^2$ or by a one-sided edge from $x$ to $y$, denoted by $(x,y) \in E^1$, or by a one-sided edge from $y$ to $x$. Other relevant notions are defined as follows.
		\begin{itemize}
			\item The \textbf{graph distance function} $d_G$ denotes the (non)-symmetric distance function in $G$, that is, $d_G(x,y)$ is the length of the shortest directed path from $x$ to $y$. If there is no such directed path, we set $d_G(x,y) = \infty$. 
			\item For $r\in \mathbb{N}$, a \textbf{(combinatorial) sphere and a ball} of radius $r$ about $x \in V$ are respectively defined as
			\begin{align*}
				S_r^G(x) &:= \{ z \in V: d_G(x,z) = r \}, \\
				B_r^G(x) &:= \{ z \in V: d_G(x,z) \le r \}.
			\end{align*}
			For simplicity, we often denote distance, spheres and balls by $d(\cdot,\cdot)$, $S_r(\cdot)$ and $B_r(\cdot)$ (that is, we drop the label $G$).
			\item A \textbf{subgraph} of $G=(V,E^1 \cup E^2)$ is a mixed graph $G_0=(V_0,E_0^1 \cup E_0^2)$ with vertex set $V_0 \subset V$ and edge set $E_0^1 \cup E_0^2$ such that\\
			1) $E_0^2 \subset E^2$, that is, every two-sided edge of $E_0$ is also a two-sided edge of $E$, and\\
			2) every one-sided edge $(x,y) \in E_0^1$ is either also a one-sided edge in $E^1$ or a two-sided edge $\{x,y\} \in E^2$. \\
			We write $G \ge G_0$ if $G_0$ is a subgraph of $G$. \textbf{Supergraphs} are defined in the opposite way and $G$ is a supergraph of $G_0$ if and only if $G_0$ is a subgraph of $G$.
		\end{itemize}
	\end{defin}
	
	
	Usually, all our mixed graphs are finite. In this case, by an enumeration of the vertices of $G$, we obtain a non-symmetric adjacency matrix $A_G$ of size $|V|$ where $(A_G)_{ij} = 1$ if there is an edge from the $i$-th vertex to the $j$-th vertex. Two-sided edges give rise to symmetric $1$-entries in $A_G$. 
	
	Next, we give the definition of weighted graphs and related notions.
	
	\begin{defin} \label{def:weightedgraph}
		A \textbf{weighted graph} $(G,P)$ is a mixed graph $G=(V,E^1\cup E^2)$ together with a \textbf{weighting scheme} $P=(p_{xy})_{x,y\in V}$, where $p_{xy}\ge 0$ represents the transition rate from $x$ to $y$. Moreover, $p_{xy}>0$ only if there is either a one-sided edge $(x,y)\in E^1$ or a two-sided edge ${x,y}\in E^2$. Further relevant notions are defined as follows.
		\begin{itemize}
			\item For a weighted graph $(G,P)$,\\
			1) a one-sided edge $(x,y)\in E^1$ is called \textbf{degenerate} if $p_{xy}=0$;\\
			2) a two-sided edge $\{x,y\}\in E^2$ is called \textbf{degenerate} if $p_{xy}=0$ or $p_{yx}=0$;\\
			3) a vertex $x\in V$ is called \textbf{degenerate} if at least one of the transition rates $p_{xy}$ corresponding to one- and two-sided edges emanating from $x$ is zero.\\ 
			A weighted graph is called \textbf{degenerate} if it has at least one degenerate edge or, equivalently, if it has at least one degenerate vertex.
			
			\item For a fixed vertex $x\in V$ of a weighted graph $(G,P)$, the transition rate $p_{xx}$ is called the \textbf{laziness} at $x$. We define $D_x := \sum_{y \neq x} p_{xy}$ to be the \emph{weighted degree} of $x$. Moreover, for any pair $x,z \in V$, we define $p_{xz}^{(2)} := \sum_y p_{xy} p_{yz}$.
			
			\item The weighting scheme $P$ is called \textbf{Markovian} if $\sum_{y\in V} p_{xy}=1$ for every vertex $x\in V$. In this case, we have $D_x=1-p_{xx}$. The set of all Markovian weighting schemes of $G$ is denoted by $\mathcal{M}_G$.
			
			\item The \textbf{induced subgraph} of $G$ by $P$ is a subgraph $G_P = (V,E_P)$ with $E_P=E_P^1 \cup E_P^2$ such that \\
			1) there is a one-sided edge $(x,y) \in E_P^1$ if and only if $p_{xy} > 0$ and $p_{yx} = 0$, and\\
			2) there is a two-sided edge $\{x,y\} \in E_P^2$ if and only if $p_{xy}, p_{yx} > 0$.\\
			Note that the weighted subgraph $(G_P,P)$ is non-degenerate, by construction.
			
			\item The corresponding distance function of $G_P$ is denoted by $d_P$ and corresponding spheres and balls are denoted by $S_r^P(x)$ and $B_r^P(x)$. 
		\end{itemize}
	\end{defin}
	
	
	
	\smallskip
	
	Bakry-\'Emery curvature rescales linearly in the weighting scheme, that is, we have $K_{\mu P,N}(x) = \mu K_{P,N}(x)$ for $\mu > 0$. For that reason, we can restrict our considerations in the case of finite graphs to Markovian weighting schemes $P$ (by a suitable rescaling and a modification of the diagonal entries $p_{xx}$ which have no influence on the Bakry-\'Emery curvature). In contrast to much other work, we do not require \emph{reversibility} of
	the Markov chain described by the matrix $P$ (which means that there is a stationary distribution $\pi: V \to (0,\infty)$ satisfying $\pi(x) p_{xy} = \pi(y) p_{yx}$ for all $x,y \in V$). A consequence of reversibility is that the spectrum of the Laplacian $\Delta_P$ in \eqref{eq:rwLap} is real, but this particular property is not relevant for our considerations. Moreover, while the normalized curvature flow in Definition \ref{def:curvflow-Q} preserves the Markovian property, it does not generally preserve reversibility.
	
	Bakry-\'Emery curvature of a vertex $x \in V$ is a local value
	and is fully determined by (the transition probabilities of the induced subgraph of) the $2$-ball $B_2^P(x)$. This follow directly from the fact that $\Gamma_2(f)(x)$ is fully determined by $B_2^P(x)$ and both $\Gamma(f)(x)$ and $\Delta f(x)$ are fully determined by $B_1^P(x)$. Note also that the upper curvature bounds $K_N^f(x)$ defined in \eqref{eq:K_f} remain the same if we replace $f$ by $c_1 f+c_2$ with $c_1 \neq 0$, that is, we have $K_N^f(x) = K_N^{c_1 f+c_2}(x)$.
	Finally, recall the following formula:
	$$ 2 \Gamma(f,g)(x) = \sum_{y} p_{xy} (f(y)-f(x))(g(y)-g(x)), $$
	which immediately implies $\Gamma(f)(x) \ge 0$ and that the curvature function $N \mapsto K_N(x)$ is monotone non-decreasing on $(0,\infty]$. Moreover, since there exists $f: V \to \mathbb{R}$ with $\Gamma(f)(x) > 0$, we have $\lim_{N \to 0} K_N(x) = - \infty$.
	
	\subsection{Lower and upper curvature bounds}
	We start with the following useful proposition.
	
	\begin{prop} \label{prop:gamma2}
		Let $(G,P)$ be a Markovian weighted graph. Then we have for any vertex $x \in V$:
		\begin{equation} \label{eq:gamma2}
			\Gamma_2(f)(x) = \left(-1 + \frac{p_{xx}}{2} \right)\Gamma(f)(x)
			+ \frac{1}{2} (\Delta f(x))^2 + \frac{1}{4} \sum_{y \neq x} p_{xy} \sum_z p_{yz} (f(z)-2f(y)+f(x))^2. 
		\end{equation}
	    Moreover, using the notation $p_{yz}\wedge p_{zy}:=\min\{p_{yz}, p_{zy}\}$, this implies the following inequality:
		\begin{equation} \label{eq:gamma2ineq}
		\Gamma_2(f)(x) \ge \left(-1 + \frac{p_{xx}}{2} +  \min_{y \in S_1^P(x)} \left(2p_{yx}+\frac{1}{2}\sum_{z\in S_1^P(x)\cap S_1^P(y)}p_{yz}\wedge p_{zy}\right) \right)\Gamma(f)(x) + \frac{1}{2} (\Delta f(x))^2. 
		\end{equation}
		In particular, if $f = d_G(x,\cdot)$, we have $\Delta f(x) = 2\Gamma(f)(x) = D_x$, and another implication of \eqref{eq:gamma2} is 
		$$ \Gamma_2(d_G(x,\cdot))(x) = \frac{1}{4} \left( D_x^2 + 3p_{xx}^{(2)} - 3p_{xx}^2 - \sum_{z \in S_2(x)} p_{xz}^{(2)} \right).
		$$
	\end{prop}
	
	\begin{proof}
		The proof of \eqref{eq:gamma2} follows the arguments given in\cite[Proof of Lemma 1]{Elworthy91}, \cite[Proof of Theorem 1.2]{LinYau2010} and \cite[Proof of Theorem 9]{JostLiu2014}. Recall that we have
		\begin{align} \label{eq: Gamma2_0}
			2\Gamma_2(f)(x) 
			&= \Delta(\Gamma(f))(x)-2\Gamma(f,\Delta f)(x).	
		\end{align}
		The first term on the right hand side can be rewritten as follows:
		\begin{align} \label{eq: Gamma2_1}
			\Delta \Gamma(f) (x) 
			&= \left( \sum_y p_{xy} \Gamma(f)(y) \right) - \Gamma(f)(x) \nonumber\\
			&= \frac{1}{2}\left( \sum_y p_{xy} \sum_z p_{yz} (f(z)-f(y))^2\right) - \frac{1}{2} \sum_y p_{xy} (f(y)-f(x))^2 \nonumber\\
			&= \frac{1}{2} \sum_y p_{xy} \sum_z p_{yz}[(f(z)-f(y))^2-(f(y)-f(x))^2] \nonumber\\
			&= \sum_{y,z} p_{xy}p_{yz} (f(z)-f(y)) (f(y)-f(x)) \\
			&+ \left( \frac{1}{2}\sum_{y,z} p_{xy} p_{yz} (f(z)-2f(y)+f(x))^2 \right) - 2\Gamma(f)(x).\nonumber
		\end{align}
		For the second term on the right hand side of \eqref{eq: Gamma2_0}, we compute
		\begin{align} \label{eq: Gamma2_2}
			2\Gamma(f,\Delta f)(x) 
			&=\sum_y p_{xy} (f(y)-f(x))(\Delta f(y) - \Delta f(x)) \nonumber\\
			&= \sum_y p_{xy} (f(y)-f(x)) \Big[ \left(\sum_z p_{yz} (f(z)-f(y))\right) -\Delta f(x) \Big] \nonumber\\
			&= \left( \sum_{y,z} p_{xy}p_{yz} (f(y)-f(x)) (f(z)-f(y)) \right) - (\Delta f(x))^2.
		\end{align}
		Plugging \eqref{eq: Gamma2_1} and \eqref{eq: Gamma2_2}  into \eqref{eq: Gamma2_0} gives 
		\begin{align*}
			2\Gamma_2(f)(x) &= -2\Gamma(f)(x) + (\Delta f(x))^2 +
			\frac{1}{2}\sum_{y,z} p_{xy} p_{yz} (f(z)-2f(y)+f(x))^2 \\
			&= -2\Gamma(f)(x) + (\Delta f(x))^2 + \frac{p_{xx}}{2} \sum_z p_{xz} (f(z)-2f(x)+f(x))^2 \\
			&+ \frac{1}{2} \sum_{y \neq x} p_{xy} \sum_z p_{yz} (f(z)-2f(y)+f(x))^2 \\
			&= (p_{xx} - 2)\Gamma(f)(x) + (\Delta f(x))^2 
			+ \frac{1}{2} \sum_{y \neq x} p_{xy} \sum_z p_{yz} (f(z)-2f(y)+f(x))^2
		\end{align*}
		which finishes the proof of the equality in \eqref{eq:gamma2}. The inequality \eqref{eq:gamma2ineq} follows from the observations that
		\[(f(z)-2f(y)+f(x))|_{z=x}=4(f(x)-f(y))^2\]
		and for $z\in S_1(x)\cap S_1(y)$
		\begin{align*}
			&(f(z)-2f(y)+f(x))^2+(f(y)-2f(z)+f(x))^2\\
			=&(f(x)-f(y))^2+(f(x)-f(z))^2+4(f(y)-f(z))^2\\
			\geq & (f(x)-f(y))^2+(f(x)-f(z))^2.
		\end{align*}
		
		\medskip
		
		In the special case $f = d_G(x,\cdot)$, we have for $y \in S_1(x)$:
		$$ \sum_z p_{yz} (f(z)-2f(y)+f(x))^2 = 4 p_{yx} + \sum_{y' \in S_1(x)} p_{yy'} = 3 p_{yx} + \left( 1 - \sum_{z \in S_2(x)} p_{yz} \right) $$
		and, therefore,
		\begin{eqnarray*} 
			\sum_{y \neq x} p_{xy} \sum_z p_{yx}(f(z)-2f(y)+f(x))^2 &=& 3 \sum_{y \neq x} p_{xy}p_{yx} + D_x - \sum_{z \in S_2(x)} \sum_{y \in S_1(x)} p_{xy} p_{yz} \\
			&=& 3 p_{xx}^{(2)} - 3 p_{xx}^2 + D_x - \sum_{z \in S_2(x)} p_{xz}^{(2)}.
		\end{eqnarray*}
		Using 
		\begin{equation} \label{eq:DeltaGammad} 
			2 \Gamma(f)(x) = \Delta f(x) = \sum_{y \neq x} p_{xy} = D_x 
		\end{equation}
		and plugging this into \eqref{eq:gamma2} yields
		\begin{eqnarray*}
			\Gamma_2(d_G(x,\cdot))(x) &=& \left(-1 + \frac{p_{xx}}{2}\right) \frac{D_x}{2} + \frac{D_x^2}{2} + 
			\frac{3p_{xx}^{(2)} - 3p_{xx}^2+D_x-\sum_{z \in S_2(x)} p_{xz}^{(2)}}{4} \\
			&=& \frac{D_x^2}{4} + \frac{3p_{xx}^{(2)} - 3p_{xx}^2-\sum_{z \in S_2(x)} p_{xz}^{(2)}}{4}.
		\end{eqnarray*}
	\end{proof}
	
	The main result of this subsection, the curvature bounds in the following theorem, is now an immediate consequence of Proposition \ref{prop:gamma2}. In particular, we have $-1 \le K_N(x) \le 2$ for $N \ge 2$, as stated in the Introduction.
	
	\begin{thm}[Curvature bounds] \label{thm:curv-bound}
		Let $(G,P)$ be a Markovian weighted graph. Then we have for all $x \in V$ and $N \ge 2$:
		$$ - 1 + \frac{p_{xx}}{2} + \min_{y \in S_1^P(x)} \left(2p_{yx}+\frac{1}{2}\sum_{z\in S_1^P(x)\cap S_1^P(y)}p_{yz}\wedge p_{zy}\right)\le K_N(x) \le K_N^{d_G(x,\cdot)}(x) \le 2 - \frac{2D_x}{N} $$
		with
		\begin{equation} \label{eq:KN0} 
			K_N^{d_G(x,\cdot)}(x) = \frac{D_x}{2} + \frac{3p_{xx}^{(2)}-3p_{xx}^2-\sum_{z \in S_2(x)} p_{xz}^{(2)}}{2D_x} - \frac{2D_x}{N}. 
		\end{equation}
		Moreover, the upper curvature bound holds for all dimensions $N \in (0,\infty]$.
	\end{thm}
	
	\begin{rmk}
		Using the Markovian property, we can remove the sum involving $z \in S_2(x)$ in formula \eqref{eq:KN0} and we obtain
		\begin{equation} \label{eq:KN0-alt} 
			K_N^{d_G(x,\cdot)}(x) = \frac{1}{2D_x} \left( 4 \sum_{y \in S_1(x)} p_{xy}p_{yx} + \sum_{y,y' \in S_1(x)} p_{xy}p_{yy'} \right) - \frac{p_{xx}}{2} - \frac{2D_x}{N}. 
		\end{equation}
		This alternative presentation will be useful later in Subsection \ref{subsec:der-curv-flow}.
	\end{rmk}

\begin{cor} \label{cor:Gamma2-nonneg}
For $x\in V$, if $\Gamma(f)(x)=0$, then $\Delta f(x)=0$ and $\Gamma_2(f)(x)\ge 0$.
\end{cor}

\subsection{$\Gamma_2$-minimizing functions}

In this subsection, we discuss another application of Proposition \ref{prop:gamma2}, namely how to determine the function values on two-spheres from the prescribed function values on one-sphere to minimize $\Gamma_2$. 

\begin{prop} \label{cor:optfunc}
	We have for any vertex $x \in V$:
	\begin{itemize}
		\item[(a)] For any set of prescribed values $f(x)$ and $f(y)$ for all $y \in S_1^P(x)$, the following values $f(z)$ for all $z \in S_2^P(x)$ are the unique choice to minimize $\Gamma_2(f)(x)$:
		\begin{equation} \label{eq:fzfy} 
			f(z) = -f(x)+\frac{2}{p_{xz}^{(2)}} \sum_{y \neq x} p_{xy}p_{yz} f(y). 
		\end{equation}
		\item[(b)] For any set of prescribed values $f(x)$ and $f(y)$ for all $y \in S_1^G(x)$, the values $f(z)$ in \eqref{eq:fzfy} for all $z \in S_2^G(x) \cap S_2^P(x)$ are the unique choice to minimize $\Gamma_2(f)(x)$.
		\item[(c)] Moreover, there exists a function $f: V \to \mathbb{R}$ with $f(x) = 0$ and
		$2 \Gamma(f)(x) = 1$, such that the corresponding upper curvature bound $K_N^f(x)$ agrees with the Bakry-\'Emery curvature:
		$$ K_N(x) = K_N^f(x), $$
	and it satisfies
	\begin{equation} \label{eq:finfest} 
		\Vert f \Vert_\infty = \sup_{z \in V} |f(z)| \le \frac{2}{\min\{ \sqrt{p_{xy}}:\, d_P(x,y) = 1 \}}.
	\end{equation}
	\end{itemize}
\end{prop}

It is important to note that the distance function $f=d_G(x,\cdot)$ satisfies the condition in (b) since we have $f(x)=0$ and $f(y)=1$ for $y\in S_1^G(x)$, and $f(z)=2$ for $z\in S_2^G(x)\cap S_2^P(x)$ satisfy \eqref{eq:fzfy}.

\begin{proof}[Proof of Proposition \ref{cor:optfunc}]
For the proof of (a) and (b), we employ \eqref{eq:gamma2} and deduce that $\Gamma_2(f)(x)$ is minimized precisely when choosing each unassigned value of $f(z)$ (that is, for $z\in S_2^P(x)$ in case (a) and for $z\in S_2(x)\cap S_2^P(x)$ in case (b)) to be 
\[f(z)=\arg\min_{t\in \mathbb{R}}\{\, \sum_{y \neq x} p_{xy}p_{yz}(t-2f(y)+f(x))^2\}.\]
The above term is a quadratic polynomial in $t$ with  strictly positive leading coefficient $p_{xz}^{(2)} = \sum_{y \neq x} p_{xy} p_{yz} > 0$ (since we only consider $z\in S_2^P(x)$ in either case), so the minimizer is uniquely given by
\[
f(z) = -f(x)+2 \frac{\sum_{y \neq x} p_{xy}p_{yz}f(y)}{\sum_{y \neq x} p_{xy}p_{yz}}. 
\]

For the statement in (c), we start with a sequence $f_n: V \to \mathbb{R}$, $\Gamma(f_n)(x)\neq 0$, such that
$$ K_N(x) = \lim_{n \to \infty} K_N^{f_n}(x). $$
The existence of such a sequence follows from Corollary \ref{cor:Gamma2-nonneg}, since functions $f$ with $\Gamma(f)=0$ are not relevant for the curvature determination. Since $K_N^{f_n}(x) = K_N^{c_1 f_n + c_2}$, $c_1 \neq 0$, we can assume w.l.o.g. that $f_n(x) = 0$ and
$2 \Gamma(f_n)(x)=1$. Moreover, we can also assume that the values $f_n(z)$ for $z \in S_2^P(x)$ are determined by the values in $f_n(y)$ for $y \in S_1^P(x)$ via \eqref{eq:fzfy} in (a), since this choice only non-increases $\Gamma_2(f_n)(x)$ and the corresponding upper bound $K_N^{f_n}(x)$. The condition $2 \Gamma(f_n)(x)=1$ implies for all $y \in S_1^P(x)$:
$$ | f_n(y) | \le 
\frac{1}{\sqrt{p_{xy}}} 
\left( \sum_{y' \neq x} p_{xy'} f_n(y')^2  \right)^{1/2} = 
\frac{(2 \Gamma(f_n)(x))^{1/2}}{\sqrt{p_{xy}}} 
= \frac{1}{\sqrt{p_{xy}}}. $$
Employing \eqref{eq:fzfy}, we conclude for all vertices $z \in S_2^P(x)$ that
$$ | f_n(z) | \le 2 \frac{\sum_{y \neq x} p_{xy}p_{yz} |f_n(y)|}{\sum_{y \neq x} p_{xy}p_{yz}} \le \max_{y\in S_1^P(x)}\frac{2}{\sqrt{p_{xy}}}. 
$$
Since all other values $f_n(w)$ with $d_P(x,w) \ge 3$ do not influence $K_N^{f_n}$, we can assume w.l.o.g. that $f_n(w) = 0$. Consequently, we have
$$ \Vert f_n \Vert_\infty \le \frac{2}{\min\{ \sqrt{p_{xy}}:\, d_P(x,y) = 1 \}}. $$
The existence of a function $f: V \to \mathbb{R}$ with $f(x) = 0$, $2\Gamma(f)(x)=1$, $K_N(x) = K_N^f(x)$ and \eqref{eq:finfest} follows now from the sequence $f_n$ by a compactness argument.
\end{proof}
	
	\begin{cor} \label{cor:optfunc2}
		Let $x \in V$ and $f: V \to \mathbb{R}$ with $f(x) = 0$ such that
		$$ K_N(x) = K_N^f(x). $$
		(The existence of such a function is guaranteed by Proposition \ref{cor:optfunc}(c).) Then the values of $f(z)$ for all $z \in S_2^P(x)$ satisfy equation \eqref{eq:fzfy}, that is, they are completely determined by the values of $f(y)$ for all $y \in S_1^P(x)$.
	\end{cor}
	
	\begin{proof}
		Let $f$ be a function such $K_N(x) = K_N^f(x)$, that is
		\begin{equation} \label{eq:curv-dim-eq} 
			\Gamma_2(f)(x) = \frac{1}{N}(\Delta f(x))^2 + K_N(x) \Gamma(f)(x). 
		\end{equation}
		Let $\tilde f$ be a function which agrees with $f$ on all vertices $w \in V$ with $d_P(x,w) \neq 2$, and which is a modification of $f$ for all $z \in S_2^P(x)$ to satisfy \eqref{eq:fzfy}. Then the right hand side of \eqref{eq:curv-dim-eq} remains unchanged when we replace $f$ by $\tilde f$. The assumption $\tilde f \neq f$ leads to
		$$ \Gamma_2(\tilde f)(x) < \Gamma_2(f)(x) = \frac{1}{N}(\Delta \tilde f(x))^2 + K_N(x) \Gamma(\tilde f)(x). $$
		This is a contradiction to
		the definition of $K_N(x)$.
	\end{proof}
	
	\begin{rmk} \label{rmk:curvsharpgen}
		Proposition \ref{cor:optfunc} and Corollary \ref{cor:optfunc2} suggest the following generalization of curvature sharpness: A vertex $x \in V$ could be called \emph{$(f,N)$-curvature sharp} if $K_N(x) = K_N^f(x)$
		with $K_N^f(x)$ defined in \eqref{eq:K_f}. 
	\end{rmk}

For later purposes, we introduce the following reformulation of Proposition \ref{cor:optfunc}(b) in terms of a $\Gamma_2$-optimal extension operator $\phi_x$ for functions defined on $B_1^G(x)$.

\begin{cor} \label{cor:optextension}
	Fix $x\in V$. Let $\phi_x:C(B_1^G(x))\to C(V)$ be the extension operator mapping functions on $B_1^G(x)$ to unique functions on $V$, defined as follows. For any $g:B_1^G(x)\to \mathbb{R}$,
	\begin{align} \label{eq:extension-defn-G}
		(\phi_x(g))(z) := 
		\begin{cases}  
			g(z) & \text{if $z \in B_1^G(x)$,} \\
			-g(x)+\frac{2}{p_{xz}^{(2)}} \sum\limits_{y \neq x} p_{xy}p_{yz} g(y) & \text{if $z \in S_2^G(x) \cap S_2^P(x)$,}\\
			0 & \text{otherwise.}
		\end{cases}
	\end{align}
	Then we have, for any $f: V \to \mathbb{R}$ with the restriction $f_0=f\vert_{B_1^G(x)}$
	\begin{equation} \label{eq:opteqGamma2} 
		\Gamma_2(f)(x) \ge \Gamma_2(\phi_x(f_0))(x). 
	\end{equation}
 Moreover, we have equality in \eqref{eq:opteqGamma2} if and only if $f$ agrees with $\phi_x(f_0)$ on $B_2^P(x)$. 
\end{cor}
	
	\subsection{Curvature reformulation via the Schur complement and the matrix $Q(x)$} \label{subsec:schur-reform}
	
	In this subsection, we revisit the core concept in \cite{CKLP-21}, which is to reformulate the Bakry-\'{E}mery curvature. With a subtle modification, we extend the reformulation of the curvature (which was defined only for a non-degenerated weighted graph) to be valid for a degenerated weighted graph.
	
	Fix a vertex $x\in V$ of a weighted graph $(G,P)$. The Laplacian $\Delta(\cdot)(x)$ and the quadratic forms $\Gamma(\cdot,\cdot)(x)$ and $\Gamma_2(\cdot,\cdot)(x)$ can be represented by a column vector $\Delta(x)$ and symmetric matrices $\Gamma(x)$ and $\Gamma_2(x)$ as follows:
	\begin{align*}
		\Delta f(x) &= \Delta(x)^\top \vec{f},\\
		\Gamma(f,g)(x) &= \vec{f}^\top \Gamma(x) \vec{g}, \\
		\Gamma_2(f,g)(x) &= \vec{f}^\top \Gamma_2(x) \vec{g},
	\end{align*}
	where $\vec{f},\vec{g}$ are representations of $f$ and $g$ as column vectors (with respect to an enumeration of the vertices). This vector and matrices have non-zero entries only in the the $2$-ball $B_2(x)$. Therefore, we use the same notation for their restrictions to $B_2(x)$. (Their explicit forms are given in \cite[Appendix A]{CKLP-21}). Moreover, the notation $\Gamma_{(2)}(x)_{W_1,W_2}$ is used for the submatrix with rows corresponding to (the vertices in) $W_1 \subset V$ and columns corresponding to $W_2$. If
	$W_1 = W_2$, we also write $\Gamma_{(2)}(x)_W$ for $\Gamma_{(2)}(x)_{W,W}$. The vertex sets chosen for $W_i$ are from the decomposition
	$$ B_2(x) = \{x\} \cup S_1(x) \cup S_2(x). $$
	Note that we use in these considerations the combinatorial spheres $S_1(x), S_2(x)$ stemming from the distance function $d_G(x,\cdot)$ and not from $d_P(x,\cdot)$ (even though the same results would hold true under this other choice). The reason for this choice is that it aligns well with the curvature flow equations which are derived later in Subsection \ref{subsec:der-curv-flow}. 
	
	\smallskip
	
	The definition of Bakry-\'Emery curvature \eqref{eq:cd-ineq} can be rephrased as the maximum value of $K$ such that $\Gamma_2(x) - \frac{1}{N}\Delta(x)\Delta(x)^\top - K \Gamma(x)$ is positive semidefinite (which we will denote by $\succeq 0$). Furthermore, since $\Delta f$, $\Gamma(f)$ and $\Gamma_2(f)$ remains unchanged under adding a constant to $f$, we may assume without loss of generality that $f(x)=0$. Equivalently, it means we are looking for the positive semidefiniteness of the matrix $(\Gamma_2(x) - \frac{1}{N}\Delta(x)\Delta(x)^\top - K \Gamma(x))_{S_1\cup S_2}$, where the row and column corresponding to the vertex $x$ are removed.
	
	
	\medskip
	
	Now we follow the arguments given in \cite[Section 1]{CKLP-21}. Writing $\Gamma_2(x)_{S_1 \cup S_2}$ as the block matrix
	$$ \Gamma_2(x)_{S_1 \cup S_2} = \begin{pmatrix} \Gamma_2(x)_{S_1} & \Gamma_2(x)_{S_1,S_2} \\ \Gamma_2(x)_{S_2,S_1} & \Gamma_2(x)_{S_2} \end{pmatrix}, $$
	we define a matrix $Q(x)$ to be the Schur complement 
	\begin{equation} \label{eq:Qx} 
		Q(x) := \Gamma_2(x)_{S_1} - \Gamma_2(x)_{S_1,S_2} \Gamma_2(x)_{S_2}^{-1} \Gamma_2(x)_{S_2,S_1} 
	\end{equation}
	in the case that $\Gamma_2(x)_{S_2}$ is positive definite (denoted by $\succ 0$). 
	
	A standard fact about Schur complements (see \cite[Lemma 2.1]{CKLP-21}) states that
	\begin{equation} \label{eq:schureq-1}
	\Big(\Gamma_2(x) - \frac{1}{N}\Delta(x)\Delta(x)^\top - K \Gamma(x)\Big)_{S_1\cup S_2} \succeq 0
	\end{equation}
	if and only if
	\begin{equation} \label{eq:schureq-2}
		Q(x) - \frac{1}{N}\Delta(x)_{S_1}\Delta(x)_{S_1}^\top - K \Gamma(x)_{S_1} \succeq 0. 
	\end{equation}
	Note also that $\Delta(x)_{S_1}$ and $\Gamma(x)_{S_1}$ simply take the form
	\begin{align} 
	\Delta(x)_{S_1} &= (p_{xy_1}, p_{xy_2}, \cdots, p_{xy_m})^\top=:{\bf p}_x, \label{eq:Delta_S1} \\
	\Gamma(x)_{S_1} &=\frac{1}{2}\diag(p_{x y_1},\dots,p_{x y_m})=\frac{1}{2}\diag({\bf p}_x). \label{eq:Gamma_S1}
	\end{align}
	where $S_1(x) = \{y_1, y_2, \dots, y_m\}$.
	
	Fortunately, the equivalence ``\eqref{eq:schureq-1} $\Leftrightarrow$ \eqref{eq:schureq-2}''extends to the case when $\Gamma_2(x)_{S_2}$ is only positive semidefinite; see Proposition \ref{prop:curvreform} below. Recall that
	we have (see \cite[(A.8) and (A.9)]{CKLP-21})
	\begin{equation} \label{eq:Gamma2-S2-diag}
	\Gamma_2(x)_{S_2} = \frac{1}{4} {\rm{diag}}(p_{xz_1}^{(2)},\dots,p_{xz_n}^{(2)})\succeq 0		
	\end{equation}
	with $S_2(x) = \{z_1,\dots,z_n\}$.
	In the degenerate case, some of the diagonal entries $p_{xz_i}^{(2)}$ may be zero, in which case \cite[Lemma 2.1]{CKLP-21} cannot be directly applied. Instead, we use \cite[Theorem 1(i)]{Alb-69} which implies equivalence of \eqref{eq:schureq-1} and \eqref{eq:schureq-2}, where the inverse $\Gamma_2(x)_{S_2}^{-1}$ in the defining equation \eqref{eq:Qx} of $Q(x)$ is replaced by the pseudoinverse $\Gamma_2(x)_{S_2}^\dagger$, under the additional assumption 
	\begin{equation} \label{eq:schur-cond} 
		\Gamma_2(x)_{S_2} \Gamma_2(x)_{S_2}^\dagger \Gamma_2(x)_{S_2, S_1} = \Gamma_2(x)_{S_2, S_1}. 
	\end{equation}
	Note that $\Gamma_2(x)_{S_2}^\dagger = \diag(q_1,\dots,q_n)$
	with
	\begin{equation} \label{eq:qjqzj} 
		q_j = q_{z_j} = \begin{cases} \frac{4}{p_{xz_j}^{(2)}}, & \text{if $p_{xz_j}^{(2)} > 0$,} \\
			0, & \text{if $p_{xz_j}^{(2)} = 0$.} \end{cases} 
	\end{equation}
	The assumption \eqref{eq:schur-cond} is easily
	verified by the fact that $p_{xz}^{(2)} = \sum_{y \in V} p_{xy}p_{yz} = 0$ implies (see \cite[A.6]{CKLP-21})
	$$ \Gamma_2(x)_{yz} = - \frac{1}{2} p_{xy}p_{yz} = 0 $$
	for all $y \in S_1(x)$ and $z \in S_2(x)$. The matrix
	$Q(x)$ defined via the pseudoinverse of $\Gamma_2(x)_{S_2}$
	has then the following explicit entries (see \cite[(A.11) and (A.12)]{CKLP-21}) for
	$y,y_i,y_j \in S_1(x)$:
	\begin{multline} \label{eq:Qepsx1}
		Q(x)_{yy} = \frac{1}{2} p_{xy}^2 + \frac{3}{4} p_{xy} p_{yx} - \frac{D_x}{4} p_{xy} + \frac{3}{4} p_{xy} \sum_{z \in S_2(x)} p_{yz} \\
		+ \frac{1}{4} \sum_{\substack{y' \in S_1(x)\\y' \neq y}} (3p_{xy}p_{yy'} + p_{xy'}p_{y'y})
		- \frac{1}{4} \sum_{z \in S_2(x)} p_{xy}^2p_{yz}^2\, q_z
	\end{multline}
	and
	\begin{equation} \label{eq:Qepsx2}
		Q(x)_{y_iy_j} = \frac{1}{2} p_{xy_i}p_{xy_j} - \frac{1}{2}p_{x y_i}p_{y_iy_j} - \frac{1}{2} p_{xy_j}p_{y_jy_i} - \frac{1}{4} \sum_{z \in S_2(x)} p_{xy_i}p_{y_iz}p_{xy_j}p_{y_jz}\, q_z
	\end{equation}
	using the factors $q_z$ given by \eqref{eq:qjqzj} in the last sums on the right hand side of \eqref{eq:Qepsx1} and \eqref{eq:Qepsx2}.  
	
	\medskip
	
	The above considerations in this more general case imply the following important curvature reformulation result.
	
	\begin{prop} \label{prop:curvreform} 
		The matrix $\Gamma_2(x)_{S_2}$ is positive semidefinite
		and we have, for any $K \in \mathbb{R}$,
		\begin{equation} \label{eq:curvcond1} 
			\left( \Gamma_2(x) - \frac{1}{N} \Delta(x)\Delta(x)^\top - K \Gamma(x) \right)_{S_1 \cup S_2} \succeq 0 
		\end{equation}
		if and only if
		$$ Q(x) - \frac{1}{N} \Delta(x)_{S_1} \Delta(x)_{S_1}^\top - K \Gamma(x)_{S_1} \succeq 0, $$ 
		with 
		\begin{equation} \label{eq:Qeps} 
			Q(x) = \Gamma_2(x)_{S_1} - \Gamma_2(x)_{S_1,S_2} \Gamma_2(x)_{S_2}^\dagger \Gamma_2(x)_{S_2,S_1}, 
		\end{equation}
		where $\Gamma_2(x)_{S_2}^\dagger$ is the pseudoinverse of $\Gamma_2(x)_{S_2}$.
		
		\smallskip
		
		In particular, Bakry-\'Emery curvature $K_N(x)$
		is the maximum of all $K \in \mathbb{R}$ satisfying
		$$ Q(x) - \frac{1}{N} \Delta(x)_{S_1} \Delta(x)_{S_1}^\top - K \Gamma(x)_{S_1} \succeq 0. $$
	\end{prop}
	
	
	We will need later Proposition \ref{prop:vKNf} below for proofs of equivalent characterisations of curvature sharp vertices. It is based on the following fundamental fact about the Schur complement, whose verification is a straightforward calculation.
	
\begin{lemma} \label{lem:Gamma2-schur} 
Let $S_1(x) = \{y_1,\dots,y_m\}$, $S_2(x) = \{z_1,\dots,z_n\}$, and ${\bf{v}} \in \mathbb{R}^m$ and ${\bf{w}} \in \mathbb{R}^n$ be arbitrary vectors. Then we have
\begin{equation} \label{eq:schur-eq-0}
	\Gamma_2(x)_{S_1 \cup S_2} \begin{pmatrix} {\bf{v}} \\ -\Gamma_2(x)_{S_2}^\dagger \Gamma_2(x)_{S_2,S_1}{\bf{v}} \end{pmatrix}
	= \begin{pmatrix} Q(x) {\bf{v}} \\ 0\end{pmatrix}, 
\end{equation}
and
\begin{equation} \label{eq:schur-eq}
	({\bf{v}}^\top, {\bf{w}}^\top)\, \Gamma_2(x)_{S_1 \cup S_2} \begin{pmatrix} {\bf{v}} \\ {\bf{w}} \end{pmatrix} = {\bf{v}}^\top Q(x) {\bf{v}} + \widetilde {\bf{w}}^\top \Gamma_2(x)_{S_2} \widetilde {\bf{w}} 
	\ge {\bf{v}}^\top Q(x) {\bf{v}} 
\end{equation}
with 
\begin{equation*} \label{eq:schur-eq2} 
	\widetilde {\bf{w}} := {\bf{w}} + \Gamma_2(x)_{S_2}^\dagger \Gamma_2(x)_{S_2,S_1} {\bf{v}}. 
\end{equation*} 
\end{lemma}

\begin{proof}
The equation \eqref{eq:schur-eq-0} follows from a straightforward matrix multiplication:
\begin{align*}
	\begin{pmatrix}
		(\Gamma_2)_{S_1} & (\Gamma_2)_{S_1,S_2} \\
		(\Gamma_2)_{S_2,S_1} & (\Gamma_2)_{S_2}
	\end{pmatrix}
	\begin{pmatrix} {\bf{v}} \\ -(\Gamma_2)_{S_2}^\dagger (\Gamma_2)_{S_2,S_1}{\bf{v}} \end{pmatrix}
	=
	\begin{pmatrix} (\Gamma_2)_{S_1}{\bf{v}}-(\Gamma_2)_{S_1,S_2}(\Gamma_2)_{S_2}^\dagger(\Gamma_2)_{S_2,S_1}{\bf{v}} \\ (\Gamma_2)_{S_2,S_1}{\bf{v}}-(\Gamma_2)_{S_2}(\Gamma_2)_{S_2}^\dagger(\Gamma_2)_{S_2,S_1}{\bf{v}}\end{pmatrix}
	=
	\begin{pmatrix}
		Q{\bf{v}}\\0
	\end{pmatrix},
\end{align*}
where we omit $x$ for simplicity.
From \eqref{eq:schur-eq-0}, we have
\begin{equation} \label{eq:extension-matrix-3}
	\Gamma_2(x)_{S_1 \cup S_2} \begin{pmatrix} {\bf{v}} \\ {\bf{w}} \end{pmatrix}
	= \begin{pmatrix} Q(x) {\bf{v}} \\ 0\end{pmatrix}
	+ \Gamma_2(x)_{S_1 \cup S_2}\begin{pmatrix} 0 \\ \widetilde {\bf{w}}\end{pmatrix},
\end{equation}
where $\widetilde {\bf{w}} = {\bf{w}} + \Gamma_2(x)_{S_2}^\dagger \Gamma_2(x)_{S_2,S_1} {\bf{v}}$. Left-multiplication of the equation \eqref{eq:extension-matrix-3} with the row vector $({\bf{v}}^\top, {\bf{w}}^\top)$ yields
\begin{align*}
	({\bf{v}}^\top, {\bf{w}}^\top)\, \Gamma_2(x)_{S_1 \cup S_2} \begin{pmatrix} {\bf{v}} \\ {\bf{w}} \end{pmatrix} 
	&= {\bf{v}}^\top Q(x) {\bf{v}} + ({\bf{v}}^\top, {\bf{w}}^\top)\, \Gamma_2(x)_{S_1 \cup S_2}\begin{pmatrix} 0 \\ \widetilde {\bf{w}}\end{pmatrix} \\
	&={\bf{v}}^\top Q(x) {\bf{v}}+ 
	\left[\begin{pmatrix} Q(x) {\bf{v}} \\ 0\end{pmatrix}
	+ \Gamma_2(x)_{S_1 \cup S_2}\begin{pmatrix} 0 \\ \widetilde {\bf{w}}\end{pmatrix} \right]^\top \begin{pmatrix} 0 \\ \widetilde {\bf{w}}\end{pmatrix} \\
	&= {\bf{v}}^\top Q(x) {\bf{v}} + \widetilde {\bf{w}}^\top \Gamma_2(x)_{S_2} \widetilde {\bf{w}}.
\end{align*}
The inequality in \eqref{eq:schur-eq} follows from $\Gamma_2(x)_{S_2}=\frac{1}{4} {\rm{diag}}(p_{xz_1}^{(2)},\dots,p_{xz_n}^{(2)})\succeq 0$.
\end{proof}

\begin{lemma} \label{lem:Gamma2-ext-Q}
Let $S_1(x) = \{y_1,\dots,y_m\}$ and $S_2(x) = \{z_1,\dots,z_n\}$. Let $f: B_1^G(x) \to \mathbb{R}$ with $f(x)=0$, and consider its extension $\phi_x^G(f):V\to \mathbb{R}$ defined as in \eqref{eq:extension-defn-G}. Let ${\bf{v}} \in \mathbb{R}^m$ and ${\bf{w}} \in \mathbb{R}^n$ denote the vectors corresponding to $\phi_x^G(f)$ restricted to $S_1^G(x)$ and $S_2^G(x)$, respectively, that is,
\begin{align*}
{\bf{v}} &:=(f(y_1),\dots,f(y_m))^\top,\\
{\bf{w}} &:=\left((\phi_x^G(f))(z_1),\dots,(\phi_x^G(f))(z_n)\right)^\top.
\end{align*}
Then we have
\begin{equation} \label{eq:optext-vector}
{\bf{w}}=-\Gamma_2(x)_{S_2}^\dagger \Gamma_2(x)_{S_2,S_1}{\bf{v}}.
\end{equation}
Furthermore, for any function $g:V\to \mathbb{R}$ with $g(x)=0$, we have
\begin{equation} \label{eq:QequalsGamma2}
	\Gamma_2(\phi_x^G(f),g)(x) = 
	\vec{f}_{S_1}^\top Q(x) \vec{g}_{S_1}
\end{equation}
\end{lemma}

\begin{proof}
Recalling the formulae of $\Gamma_2(x)_{S_2,S_1}$ and $\Gamma_2(x)_{S_2}^\dagger$ from \eqref{eq:Gamma2-S2-diag} and \eqref{eq:qjqzj}, we derive  
	\begin{align*}
		-\Gamma_2(x)_{S_2}^\dagger \Gamma_2(x)_{S_2,S_1}{\bf{v}}=
		\begin{pmatrix} q_{z_1} & & \\
			& \ddots & \\
			& & q_{z_n}
		\end{pmatrix}
		\cdot \frac{1}{2}\begin{pmatrix} p_{x y_1}p_{y_1 z_1} & \dots & p_{x y_m}p_{y_m z_1} \\
			\vdots & & \vdots \\
			p_{x y_1}p_{y_1 z_n} & \dots & p_{x y_m}p_{y_m z_n}
		\end{pmatrix} \begin{pmatrix} f(y_1) \\ \vdots \\ f(y_m) \end{pmatrix}.
	\end{align*}
	Its $j$-th entry corresponding to $z_j\in S_2^G(x)$ is equal to
	\begin{align*}
		(-\Gamma_2(x)_{S_2}^\dagger \Gamma_2(x)_{S_2,S_1}{\bf{v}})_j=\begin{cases}
			\displaystyle{\frac{2}{p_{xz_j}^{(2)}} \sum_{i=1}^m p_{xy_i}p_{y_iz_j} f(y_i)} & \text{if $z_j \in S_2^P(x)\cap S_2^G(x)$,} \\
			0 & \text{otherwise,}
		\end{cases}
	\end{align*} 
	which is precisely $(\phi_x^G(f))(z_j)$ as defined in \eqref{eq:extension-defn-G}. Thus  \eqref{eq:optext-vector} is proved.
	
	\smallskip
	
	Next, by combining \eqref{eq:schur-eq-0} and \eqref{eq:optext-vector}, we derive
	\begin{equation*}
		\Gamma_2(\phi_x^G(f),g)(x)=\vec{g}^\top \Gamma_2(x)_{S_1\cup S_2} \begin{pmatrix} {\bf{v}} \\ -\Gamma_2(x)_{S_2}^\dagger \Gamma_2(x)_{S_2,S_1}{\bf{v}} \end{pmatrix} = \vec{g}^\top \begin{pmatrix} Q(x) {\bf{v}} \\ 0\end{pmatrix}=\vec{g}_{S_1}^\top Q(x)\vec{f}_{S_1},
	\end{equation*}
	and thus prove \eqref{eq:QequalsGamma2}.
\end{proof}
	
The matrix $Q(x)$ has the following interesting characterisation in terms of
$\Gamma_2$:
	
\begin{prop} \label{prop:Qcharacterisation}
    Let $S_1(x) = \{y_1,\dots,y_m\}$. The matrix $Q(x)$ is the unique symmetric matrix satisfying
    $$ {\bf{v}}^\top Q(x) {\bf{v}} = \min \Gamma_2(f), $$
    for all ${\bf{v}} = (v_1,\dots,v_m)$, where the minimum runs over all $f: V \to \mathbb{R}$ with $f(x) = 0$ and $f(y_j) = v_j$ with $1 \le j \le m$.
\end{prop}
	
\begin{proof}
  It is a direct consequence of Corollary \ref{cor:optextension} and  \eqref{eq:QequalsGamma2} that $Q$ has the desired property. Moreover, symmetric bilinear forms are uniquely determined by their diagonal, finishing the proof.
\end{proof}	
	
\begin{prop} \label{prop:vKNf}
	Let $(G,P)$ be a Markovian weighted graph, $x \in V$ a vertex, $S_1(x) = \{y_1,\dots,y_m\}$, and $S_2(x) = \{z_1,\dots,z_n\}$. Let $f: V \to \mathbb{R}$ be a function with $f(x) = 0$ and $\Gamma(f)(x) \neq 0$. Then
	\begin{itemize}
		\item[(a)] $$\vec{f}^\top \left( \Gamma_2(x) - \frac{1}{N} \Delta(x)\Delta(x)^\top - K_N^f(x) \Gamma(x)\right)_{S_1 \cup S_2} \vec{f} = 0,$$ where $\vec{f}$ denotes the vector $ (f(y_1),\dots,f(y_m),f(z_1),\dots,f(z_n))^\top.$
			
		\item [(b)] 
		If $f(z)$ for $z \in S_2^G(x)\cap S_2^P(x)$ are related to the values $f(y) \in S_1^G(x)$ as in Proposition \ref{cor:optfunc}(b), then 
		\[
			 {\bf{v}}^\top Q(x){\bf{v}} = {\bf{v}}^\top \left( \frac{1}{N} \Delta(x)_{S_1} \Delta(x)_{S_1}^\top + K_N^f(x) \Gamma(x)_{S_1} \right) {\bf{v}}, 
		\]
		where $ {\bf{v}}$ denotes the vector $(f(y_1),\dots,f(y_m))^\top$.
	\end{itemize}
\end{prop}
	
	\begin{proof}
		The statement (a) follows directly from the definition of $K_N^f(x)$. For (b),
		denote the vector
		$$ {\bf{w}} := (f(z_1),\dots,f(z_n))^\top, $$
		that is, $\vec{f}^\top = ({\bf{v}}^\top,{\bf{w}}^\top)$. 
		Without loss of generality, assume $f=0$ outsides $B_2^P(x)$ (as the statement (b) is not affected). Then $f$ agrees with $\phi_x^G(f_0(x))$ where $f_0:=f|_{B_1^G(x)}$. In particular, the vector ${\bf{w}}$ agrees with
		\begin{align*}
		\left((\phi_x^G(f_0))(z_1),\dots,(\phi_x^G(f_0))(z_n)\right)^\top,
		\end{align*}
		and Lemma \ref{lem:Gamma2-ext-Q} implies that
		${\bf{w}}=-\Gamma_2(x)_{S_2}^\dagger \Gamma_2(x)_{S_2,S_1}{\bf{v}}$. In view of \eqref{eq:schur-eq}, we have
		\begin{equation*}
			({\bf{v}}^\top, {\bf{w}}^\top)\, \Gamma_2(x)_{S_1 \cup S_2} \begin{pmatrix} {\bf{v}} \\ {\bf{w}} \end{pmatrix} = {\bf{v}}^\top Q(x) {\bf{v}}, 
		\end{equation*}
		since $\widetilde {\bf{w}} = {\bf{w}} + \Gamma_2(x)_{S_2}^\dagger \Gamma_2(x)_{S_2,S_1} {\bf{v}}=0$. Combining the above equation with the statement (a), we can conclude that
		\[ {\bf{v}}^\top \left( Q(x) - \frac{1}{N}\Delta(x)\Delta(x)^\top - K_N^f(x) \Gamma(x) \right) {\bf{v}} = 0,\]
		which finishes the proof.		
\end{proof}

\medskip
	
Let us finish this subsection with some remarks about the relevance of the matrix $Q(x)$. In the case of a \emph{non-degenerate} vertex $x \in V$, the matrix $Q(x)$ gives rise to the family of so-called curvature matrices 
\begin{equation} \label{eq:AN-def}
	A_N(x)=2 {\rm{diag}}({\bf{v}}_0)^{-1} Q(x) {\rm{diag}}({\bf{v}}_0)^{-1} - \frac{2}{N} {\bf{v}}_0(x) {\bf{v}}_0(x)^\top,
\end{equation}
where ${\bf{v}}_0(x)=(\sqrt{p_{xy_1}},\dots, \sqrt{p_{xy_m}})^\top$ and $S_1(x)=\{y_1,...,y_m\}$. As already mentioned in the Introduction, these curvature matrices $A_N(x)$ are symmetric matrices whose smallest eigenvalues agree with the Bakry-\'Emery curvatures of the vertex $x$:
	\[ K_N(x) = \lambda_{\min}(A_N(x)), \]
and by Rayleigh quotient characterisation,
\[ K_N(x) \le \frac{{\bf{v}}_0(x)^\top A_N(x){\bf{v}}_0(x)}{{\bf{v}}_0(x)^\top {\bf{v}}_0(x)} =:K_N^0(x),\]
where ${\bf{v}}_0(x) = (\sqrt{p_{xy_1}},\dots,\sqrt{p_{xy_m}})^\top$.

\medskip

We like to emphasize that when $x$ is a degenerate vertex of $G$ (that is, some $p_{xy_i}$ vanishes), $A_N(x)$ is undefined; however, $Q(x)$ is always well defined. Since curvature depends only on the weighting scheme $P$ and not on the graph $G$, it is possible to still define the matrices $A_N(x)$ for degenerate vertices $x$ of $G$ by changing to the subgraph $(G_P,P)$, in which all vertices are non-degenerate. In this case, however, the size $|S_1^P(x)|$ of the matrices $A_N(x)$ is smaller than the size of the matrix $Q(x)$ for the graph $(G,P)$.

\subsection{Relations between upper curvature bounds}
	
In this subsection, we first prove the agreement of the upper curvature bounds $K_N^{d_P(x,\cdot)}(x)$ and $K_N^0(x)$. This implies that, in the non-degenerate case, the definition of curvature sharpness in the current paper agrees with the curvature sharpness definition $K_N(x)=K_N^0(x)$ introduced in \cite{CKLP-21}. Next, we prove the inequality of the upper curvature bounds $K_N^{d_{G_0}(x,\cdot)}(x) \le K_N^{d_G(x,\cdot)}(x)$ for a subgraph $(G_0,P)$ of $(G,P)$, and give the exact condition when these bounds agree. This fact will be crucial for our study of curvature sharpness in the next section.

\begin{prop}[Agreement of upper curvature bounds - non-degenerate case] \label{prop:coinc-curv-bounds}
	Consider a non-degenerate Markovian weighted graph $(G_P,P)$. Then we have for all $x \in V$
	and all $N \in (0,\infty]$,
	\[	K_N^{d_P(x,\cdot)}(x) = K_N^0(x). \]
	
\end{prop}
	
\begin{proof}
Recall from the definitions of $K_N^0(x)$ and $A_N(x)$ in the non-degenerate setting (where the matrices, $Q(x)$, $A_N(x)$ and vector ${\bf v}_0(x)$ belong to the graph $(G_P,P)$ and have the size of $m=|S_1^P(x)|$): 
\begin{align} \label{eq:KN0-proof-eq1}
	K_N^0(x) 
	&= \frac{{\bf v}_0^\top \Big( 2 {\rm{diag}}({\bf{v}}_0)^{-1} Q(x) {\rm{diag}}({\bf{v}}_0)^{-1} - \frac{2}{N} {\bf{v}}_0 {\bf{v}}_0^\top \Big) {\bf v}_0}{{\bf v}_0^\top{\bf v}_0} \nonumber \\
	&=\frac{2}{D_x}{\bf 1}_m^\top Q(x) {\bf 1}_m-\frac{2D_x}{N}
\end{align}
			
We apply Proposition \ref{prop:vKNf}(b) with the function $f = d_P(x,\cdot)$, and derive that
\begin{align} \label{eq:KN0-proof-eq2}
	{\bf 1}_m^\top Q(x) {\bf 1}_m 
	&=
	{\bf 1}_m^\top \left( \frac{1}{N} \Delta(x)_{S_1} \Delta(x)_{S_1}^\top + K_N^{d_P(x,\cdot)}(x) \Gamma(x)_{S_1} \right) {\bf 1}_m \nonumber \\
	&= \frac{D_x}{N} + \frac{K_N^{d_P(x,\cdot)}(x)D_x}{2},
\end{align}
due to $\Delta(x)_{S_1}={\bf p}_x$ and $\Gamma(x)_{S_1}=\frac{1}{2} \diag({\bf p}_x)$. Combining \eqref{eq:KN0-proof-eq1} and \eqref{eq:KN0-proof-eq2} yields $K_N^0(x)=K_N^{d_P(x,\cdot)}(x)$.
		
	\end{proof}
	
The following theorem provides the relation between these upper bounds with respect to subgraphs.
	
	\begin{thm} \label{thm:uppbds_subgraphs}
		Let $(G,P)$ be a Markovian weighted graph and $G_0=(V,E_0)$ be a mixed subgraph of $G$ with $G_P \le G_0 \le G$. Then we have for any vertex $x \in V$ and for any dimension $N\in (0,\infty]$:
		\begin{equation} \label{eq:curvhier}
			K_N^{d_P(x,\cdot)} \le K_N^{d_{G_0}(x,\cdot)}(x) \le K_N^{d_G(x,\cdot)}(x).
		\end{equation} 
		Moreover, the equality $K_N^{d_{G_0}(x,\cdot)} = K_N^{d_{G}(x,.)}(x)$ holds if and only if the following condition holds:
		\begin{equation} \label{eq:G=P_on_B2}
			d_{G_0}(x,\cdot)=d_G(x,\cdot) \text{ on } B_2^P(x).
		\end{equation}
	\end{thm}
	
	\begin{proof}
		Note that Bakry-\'Emery curvature is independent of the topology of the graph and that the values $\Gamma(f)(x)$ and $\Delta f(x)$ are the same for all three functions $f = d_P(x,\cdot), d_{G_0}(x,\cdot)$ and $d_G(x,\cdot)$. For the proof of \eqref{eq:curvhier}, it suffices therefore to show
		\begin{equation} \label{eq:gammahier} 
			\Gamma_2(d_P(x,\cdot))(x) \le \Gamma_2(d_{G_0}(x,\cdot))(x) \le \Gamma_2(d_G(x,\cdot))(x), 
		\end{equation}
		and we only need to investigate the term
		\begin{equation} \label{eq:critterm} 
			\sum_{y \neq x} \sum_z p_{xy} p_{yz} (f(z)-2f(y))^2 
		\end{equation}
		in \eqref{eq:gamma2} for the respective distance functions $f$. Note that $d_P(x,z) = 1$ automatically implies also $d_{G_0}(x,z) = 1$ and $d_G(x,z) = 1$, and the term \eqref{eq:critterm} simplifies for all three functions to
		\begin{equation} \label{eq:critterm2} 
			\sum_{y:\, d_P(x,y) = 1} p_{xy}\left( 4 p_{yx} + \sum_{z:\,  d_P(x,z) = 1} p_{yz} + \sum_{z:\, d_P(x,z) = 2} p_{yz}(f(z)-2)^2 \right).
		\end{equation}
		The inequality \eqref{eq:gammahier} follows then from an observation that $1 \le d_G(x,z) \le d_{G_0}(x,z) \le 2$ for all vertices $z \in S_2^P(x)$.
		
		Moreover, $K_N^{d_{G_0}(x,.)} = K_N^{d_{G}(x,.)}(x)$ holds if and only if $\Gamma_2(d_{G_0}(x,\cdot))(x) = \Gamma_2(d_G(x,\cdot))(x)$, which (according to \eqref{eq:critterm2}) holds if and only if 
		$$\sum_{z:\, d_P(x,z) = 2} p_{yz}(d_{G_0}(x,z)-2)^2 = \sum_{z:\, d_P(x,z) = 2} p_{yz}(d_{G}(x,z)-2)^2.$$
		This occurs exactly when $d_{G_0}(x,\cdot)=d_G(x,\cdot)$ on $B_2^P(x)$ (or otherwise, there would exist some $v\in S_2^P(x)$ with $d_{G}(x,v)=1$ and $d_{G_0}(v,z)=2$ and the above equality would never hold).
	\end{proof}
	

	\section{Analytic and geometric aspects of curvature sharp vertices}
	\label{sec:analgeomcurvsharpvert}
	
	Let us start with some background information about the curvature sharpness notion. Curvature sharpness of a vertex was originally introduced in the case of unweighted non-normalized Laplacian in \cite[Definition 1.4]{CLP-20} via an upper curvature bound based on the condition
	$$ \det \left[ \left(\Gamma_2(x)- \frac{1}{N} \Delta(x) \Delta(x)^\top - K_N(x) \Gamma(x)\right)_{\{x\} \cup S_2} \right] \ge 0. $$ 
	It was also shown in \cite[Corollary 5.10]{CLP-20} that, in this setting, curvature sharpness of a vertex $x \in V$ is equivalent to $S_1$-out regularity of $x$ (see \cite[Definition 1.5]{CLP-20}). 
	
	\smallskip
	
	Another definition of curvature sharpness was given, for non-degenerate reversible weighted graphs, in \cite[Theorem 1.5]{CKLP-21} via the Rayleigh quotient upper bound
	\eqref{eq:KN-Rayleigh}. It was noticed in \cite[Remark 4.1]{CKLP-21} that this definition generalises the earlier one given in \cite{CLP-20}. This second definition led to the characterisation that $x \in V$ is curvature sharp if and only if 
	${\bf v}_0(x)$ is an eigenvector of $A_\infty(x)$ (see \cite[Proposition 1.7(i)]{CKLP-21}).  It was also shown in \cite[Theorem 1.14]{CKLP-21} that curvature sharpness of $x \in V$ follows from $S_1$-in and $S_1$-out regularity of $x$. 
	
	\smallskip
	
	In the Introduction, we provided a third definition of curvature sharpness of a vertex $x \in V$ employing the combinatorial distance function $d_G(x,\cdot)$ in the original curvature-dimension inequality. There, a vertex $x$ is said to be \emph{curvature sharp} if we have, for some $N \in (0,\infty]$, $K_N(x) = K_N^{d_G(x,\cdot)}(x)$. This definition is inspired by \cite[Theorem 1.2]{KKRT16}, and it is still valid in the case of non-reversible degenerate weighted graphs. 
	
	\smallskip
	
	The main result in this section are the curvature sharpness equivalences listed in Theorem \ref{thm:curvsharpeq}.
	
	

\subsection{Monotonicity properties of curvature sharpness} In this subsection we investigate the behaviour of curvature sharpness under change of the dimension parameter $N \in (0,\infty]$.

\begin{prop} \label{prop:curvsharpmon}
	Let $x \in V$ and $N' \le N$. Then we have
	\begin{equation} \label{eq:curvestlow} 
		K_N(x) - \left( \frac{2D_x}{N'} - \frac{2D_x}{N} \right) \le K_{N'}(x) \le K_N(x). 
	\end{equation}
	Moreover, if $x$ is $N$-curvature sharp, then $x$ is also $N'$-curvature sharp for all dimensions $N' \le N$ and we have
	\begin{equation} \label{eq:curvdescr-curvsharp}
		K_{N'}(x) = K_N(x) - \left( \frac{2D_x}{N'} - \frac{2D_x}{N} \right). 
	\end{equation}
\end{prop}

\begin{proof}
	The monotone non-decreasing property of $N \mapsto K_N(x)$ was already mentioned earlier and is a consequence of $\Gamma(f)(x) \ge 0$ for all functions $f$. So we only need to prove the left hand inequality of \eqref{eq:curvestlow}, which is equivalent to the statement that
	$$ N \mapsto \frac{2D_x}{N} + K_N(x) $$
	is non-increasing on $(0,\infty]$. Let us first prove this monotonicity. Since we have
	$$ \frac{2D_x}{N} + K_N(x) = \inf_{\Gamma(g)(x)\neq 0} \left(\frac{2D_x}{N} + K_N^g(x) \right), $$
	it suffices to show that
	$$ N \mapsto \frac{2D_x}{N} + K_N^g(x) $$
	is monotone non-increasing on $(0,\infty]$ for all $g$ with $\Gamma(g)(x) \neq 0$. We have
	$$ \frac{2D_x}{N} + K_N^g(x) = \frac{\Gamma_2(g)(x)}{\Gamma(g)(x)} +\left( 2D_x - \frac{(\Delta g(x))^2}{\Gamma(g)(x)} \right) \frac{1}{N},
	$$
	and its monotonicity follows then from $(\Delta g(x))^2\le 2D_x\Gamma(g)(x)$ due to Cauchy-Schwarz:
	\[ (\Delta g(x))^2 = \left(\sum_{y} p_{xy} (g(y)-g(x))\right)^2 \le \left( \sum_{y} p_{xy} \right) \left( \sum_{y} p_{xy} (g(y)-g(x))^2 \right) = D_x \cdot 2\Gamma(g)(x).
	\]
	
	\medskip
	
	Assume now that $x$ is $N$-curvature sharp. Then we have for $N' \le N$ and $f = d_G(x,\cdot)$:
	\begin{multline*}
		K_{N'}(x) \stackrel{(*)}{\le} K_{N'}^f(x) = \frac{\Gamma_2(f)(x)}{\Gamma(f)(x)} - \frac{1}{N'}
		\frac{(\Delta f(x))^2}{\Gamma(f)(x)} = \frac{\Gamma_2(f)(x)}{\Gamma(f)(x)} - \frac{2D_x}{N'} = \\ \left( \frac{\Gamma_2(f)(x)}{\Gamma(f)(x)} - \frac{1}{N}\frac{(\Delta f(x))^2}{\Gamma(f)(x)} \right) + \left( \frac{2D_x}{N} - \frac{2D_x}{N'} \right) = K_N(x) - \left( \frac{2D_x}{N'} - \frac{2D_x}{N} \right). 
	\end{multline*}
	Combining this with \eqref{eq:curvestlow} implies that
	we have equality in $(*)$, that is, $x$ is $N'$-curvature sharp and we have \eqref{eq:curvdescr-curvsharp}.
\end{proof}

\begin{rmk}
	The second part of Proposition \ref{prop:curvsharpmon} does no longer hold for the generalization of curvature sharpness proposed in Remark \ref{rmk:curvsharpgen}: $(f,N)$-curvature sharpness $x$ does not necessarily imply $(f,N')$-curvature sharpness of $x$ for $N' \le N$ for general functions $f: V \to \mathbb{R}$ with $\Gamma(f)(x) \neq 0$. This result can only be derived in the special case of the distance function $f = d_G(x,\cdot)$.
\end{rmk}

The above ``monotonicity'' property of curvature sharpness raises the question whether there exist an absolute small dimension value $N_0 > 0$ such that curvature sharpness implies always $N_0$-curvature sharpness. The following theorem gives a positive answer with an optimal threshold $N_0=2$.

\begin{thm} \label{thm:curvsharp2curvsharp}
	Let $(G,P)$ be a weighted graph and $x \in V$ be a curvature sharp vertex. Then $x$ is $2$-curvature sharp.
\end{thm}

We postpone the proof of this result since it overlaps with the proof of the later Theorem \ref{thm:curvsharpeq-Q}, and we will present the combined proof of both Theorems there. Moreover, we will see from this proof that $N_0=2$ in Theorem \ref{thm:curvsharp2curvsharp} is the optimal threshold (see Remark \ref{rmk:curv-sharp-threshold} after the combined proof).


\subsection{Curvature sharpness of vertices in subgraphs and supergraphs}

Recall that Bakry-\'Emery curvature $K_N(x)$ at a vertex $x \in V$ of a weighted graph $(G,P)$ is fully determined by the weighting scheme $P$ and is independent of the graph $G$. In fact,
we have 
$$ K_N(x) = \lambda_{\min}(A_N(x)), $$
where $A_N(x)$ is a matrix of size $|S_1^P(x)|$, deduced from the non-degenerate weighted subgraph $(G_P,P)$. 

\medskip

Curvature sharpness of a vertex $x \in V$ (that is, $K_N(x)=K_N^{d_G(x,\cdot)} (x)$) depends, however, on both the weighting scheme $P$ and the topology given by the mixed graph $G=(V,E)$. It is natural to ask whether curvature sharpness is preserved under taking subgraphs or supergraphs of $G$. The next proposition states that this is the case for all mixed subgraphs $G_0$ of $G$, obtained by a removing some edges corresponding to $p_{xy} = 0$, $x \neq y$.

\begin{prop}[Curvature sharpness of sub-/supergraphs] \label{prop:curvsharpsubgraph}
Let $(G,P)$ be a Markovian weighted graph. Let $x\in V$ and $N\in (0,\infty]$. Suppose that $x$ in $N$-curvature sharp in $(G,P)$. 

\begin{itemize}
	\item[(a)] Then $x$ is also $N$-curvature sharp in $(G_0,P)$ for any subgraph $G_0$ such that $G_P\le G_0 \le G$. Moreover, $d_{G_0}(x,\cdot)=d_{G}(x,\cdot)$ on $B_2^P(x)$.
	\item[(b)] Then for any supergraph $G'\ge G$, the vertex $x$ is $N$-curvature sharp in $(G',P)$ if and only if $d_{G}(x,\cdot)=d_{G'}(x,\cdot)$ on $B_2^P(x)$.
\end{itemize}
\end{prop}

\begin{proof}
Both statements are straightforward consequences of Theorem \ref{thm:uppbds_subgraphs}: Note that we have
\[ K_N(x) \le K_N^{d_{G_0}(x,\cdot)}(x) \le K_N^{d_G(x,\cdot)}(x) \le  K_N^{d_{G'}(x,\cdot)}(x),\] 
and suppose $K_N(x)=K_N^{d_G(x,\cdot)}(x)$. Then the first two inequalities above must hold with equality. The first equality means $x$ is $N$-curvature sharp in $(G_0,P)$, and the second one implies $d_{G_0}(x,\cdot)=d_{G}(x,\cdot)$ on $B_2^P(x)$ due to Theorem \ref{thm:uppbds_subgraphs}. Lastly, $x$ is $N$-curvature sharp in $(G',P)$ if and only if the last inequality above holds with equality (since the first two are already equality), which occurs exactly when $d_{G}(x,\cdot)=d_{G'}(x,\cdot)$ on $B_2^P(x)$ due to Theorem \ref{thm:uppbds_subgraphs}.
\end{proof}

\begin{ex}
Curvature sharpness of a vertex $x \in V$ is not necessarily preserved if we change the topology of
a weighted graph $(G,P)$ to a weighted supergraph $(G',P)$. For example, the simple random walk (without laziness) $P$ on the square $G = K_2 \times K_2$ (without one-sided edges) is $\infty$-curvature sharp and we have $K_\infty(x) = 1$ for all vertices. If we 
keep this weighting scheme $P$ and add two two-sided edges to obtain the complete graph $K_4$, the original upper curvature bound $K_\infty^0(x) = 1$, given by \eqref{eq:KN0}, changes into 
$$ K_\infty^0(x) = \frac{1}{2} + \frac{3p_{xx}^{(2)}}{2} = \frac{5}{4}. $$
since there are no longer vertices $z \in S_2(x)$ in a complete graph. Since $K_N(x) = 1 - \frac{2}{N}$ and $K_N^0(x) = \frac{5}{4} - \frac{2}{N}$ for all vertices, the vertices in $K_4$ are no longer $N$-curvature sharp for any dimension $N$.  
\end{ex}



%

\subsection{Curvature sharpness equivalences} \label{subsec:curvsharpeq}

The main goal of this subsection is to show that
curvature sharpness of a vertex $x \in V$ with $S_1(x) = \{y_1,\dots,y_m\}$ is equivalent to the identity
\begin{equation} \label{eq:QKinf} 
	Q(x) {\bf{1}}_m = \frac{K_\infty^{d(x,\cdot)}(x)}{2} {\bf{p}}_x, 
\end{equation}
where $Q(x)$ is our Schur complement of $\Gamma_2(x)$ defined as in \eqref{eq:Qx}, and 
$$ {\bf{p}}_x := (p_{x y_1},\dots,p_{x y_m})^\top.$$

The relevance of this fact is that it implies immediately that the \emph{stationary solutions} of the normalized curvature flow given in Definition \ref{def:curvflow-Q} (see also \eqref{eq:curv-flow-gen-Q} with $C_x(t) = K_{P(t),\infty}^{d_G(x,\cdot)}(x)$) are precisely the \emph{curvature sharp} weighting schemes $P \in \mathcal{M}_G$.

Let us start with the following lemma which describes curvature sharpness at a fixed dimension by employing the curvature characterization (Proposition \ref{prop:curvreform}) and a crucial fact about the Schur complement (Proposition \ref{prop:vKNf}).

\begin{lemma} \label{lem:curvsharp_M}
	Let $x\in V$ with $S_1(x)=\{y_1,y_2,\ldots,y_m\}$, and let $N\in (0,\infty]$. We define the following matrix
	\begin{equation} \label{eq:matrixM}
		M_N(x):=Q(x)-\frac{1}{N}\Delta(x)_{S_1}\Delta(x)_{S_1}^\top-K_N^{d(x,\cdot)}(x)\Gamma(x)_{S_1}.
	\end{equation}
	Then 
	\begin{itemize} 		
	\item[(a)] ${\bf{1}}_m^\top M_N(x) {\bf{1}}_m = 0$. In particular, $\frac{K_\infty^{d_G(x,\cdot)}(x) D_x}{2}= {\bf{1}}_m^\top Q(x) {\bf{1}}_m$.
	\item[(b)] The vertex $x$ is $N$-curvature sharp if and only if $M_N(x)$ is positive semidefinite.
	\item[(c)] If $x$ is $N$-curvature sharp, then the vector ${\bf 1}_m$ is in the kernel of $M_N(x)$.
	\end{itemize}  
\end{lemma}

We would like to emphasize that the vector ${\bf 1}_m$ here represents the vector $(f(y_1)\ f(y_2) \ldots f(y_m))^\top$ in the special case that $f=d_G(x,\cdot)$.

\begin{proof}[Proof of Lemma \ref{lem:curvsharp_M}]
The first statement of (a) follows from Proposition \ref{prop:vKNf}(b) with $f=d_G(x,\cdot)$. By choosing $N=\infty$ and using $\Gamma(x)=\frac{1}{2}\diag({\bf{p}}_x)$, we obtain the second identity.		
The statement (b) follows from Proposition \ref{prop:curvreform} and the fact that $K_N(x)\le K_N^{d(x,\cdot)}(x)$. For (c), we conclude from ${\bf{1}}_m^\top M_N(x) {\bf{1}}_m = 0$ and $M_N(x)\succeq 0$ (due to $x$ being $N$-curvature sharp) that $M_N(x){\bf 1}_m = 0$.
\end{proof}

Now we are ready to prove the characterization of curvature sharpness in \eqref{eq:QKinf}, which is independent of the dimension parameter.

\begin{thm} \label{thm:curvsharpeq-Q}
	Let $(G,P)$ be a Markovian weighted graph. Let $x \in V$ with $S_1(x) = \{y_1,\dots,y_m\}$. Then $x$ is curvature sharp if and only if 	\[\left(Q(x)-K_\infty^{d(x,\cdot)}(x)\Gamma(x)_{S_1}\right) {\bf 1}_m=0,\]
	or equivalently,
	\[Q(x) {\rm 1}_m = \frac{1}{2}K_\infty^{d(x,\cdot)}(x) {\bf p}_x.\]
\end{thm}

The proof of this theorem leads directly to the statement of Theorem \ref{thm:curvsharp2curvsharp} that curvature sharpness of a vertex implies $2$-curvature sharpness. Therefore, we combine both proofs.

\begin{proof}[Proof of Theorems \ref{thm:curvsharpeq-Q} and \ref{thm:curvsharp2curvsharp}]
	Recalling $K_N^{d(x,\cdot)}(x)=K_\infty^{d(x,\cdot)}(x)-\frac{2D_x}{N}$ (from \eqref{eq:KN0}), the matrix $M_N(x)$ in \eqref{eq:matrixM} can be rewritten as
	\begin{align*}
		M_N(x) &=\left(Q(x)-K_\infty^{d(x,\cdot)}(x)\Gamma(x)_{S_1}\right) + \frac{1}{N}\left( 2D_x\Gamma(x)_{S_1}-\Delta(x)_{S_1}\Delta(x)_{S_1}^\top\right) \\
		&=: M_\infty(x)+\frac{1}{N}R(x).
	\end{align*}
	Using $\Delta(x)_{S_1}={\bf p}_x^\top$ and $2\Gamma(x)_{S_1}=\diag({\bf p}_x)$  (see \eqref{eq:Delta_S1} and \eqref{eq:Gamma_S1}), we can derive that $R(x){\bf 1}_m=0$.
	
	\medskip
	
	To prove the forward implication, we assume that $x$ is $N$-curvature sharp for some $N\in (0,\infty]$. Lemma \ref{lem:curvsharp_M} implies that ${\bf 1}_m\in \ker M_N(x)$. Since ${\bf 1}_m \in \ker R(x)$, we also have ${\bf 1}_m \in \ker M_\infty(x)$, that is,
	
	\[\left(Q(x)-K_\infty^{d(x,\cdot)}(x)\Gamma(x)_{S_1}\right) {\bf 1}_m=0,\]
	or equivalently,
	\[Q(x) {\rm 1}_m = \frac{1}{2}K_\infty^{d(x,\cdot)}(x) {\bf p}_x.\]
	
	To prove the reverse implication, we assume the identity $\left(Q(x)-K_\infty^{d(x,\cdot)}(x)\Gamma(x)_{S_1}\right) {\bf 1}_m=0$. Together with $R(x){\bf 1}_m=0$, we have $M_N(x){\bf 1}_m=0$ for all $N\in (0,\infty]$, which means the entries in each row of $M_N(x)$ sum up to zero. We will show that, for small enough $N\in (0,\infty]$, the off-diagonal entries of $M_N(x)$ are non-positive. From then, we can conclude that $M_N(x)$ is diagonally dominant with non-negative diagonal entries and hence it is positive semidefinite. Lemma \ref{lem:curvsharp_M} will then imply that $x$ is $N$-curvature sharp for those small $N$.
	In order to compute the off-diagonal entries of $M_N(x)$, we recall those of $Q(x)$ from \eqref{eq:Qepsx2} and observe that for any $y,y'\in S_1(x)$ with $y\neq y'$, we have
	\begin{align*}
		\left(\Delta(x)_{S_1}\Delta(x)_{S_1}^\top\right)_{yy'}=p_{xy}p_{xy'}, \text{ and } \left(\Gamma(x)_{S_1}\right)_{yy'}=0.
	\end{align*}
	Thus 
	\[(M_N(x))_{yy'}=\left(\frac{1}{2}-\frac{1}{N}\right)p_{xy}p_{xy'}-\frac{1}{2}p_{xy}p_{yy'}-\frac{1}{2}p_{xy'}p_{y'y}-\frac{1}{4}\sum_{z\in S_2(x)}p_{xy}p_{yz}p_{xy'}p_{y'z}q_z.\]
	
	In particular when $N\le 2$, every off-diagonal entry of the matrix $M_N(x)$ is non-positive as desired. Therefore, $x$ must be $N$-curvature sharp for all $N\le 2$. Thus Theorem \ref{thm:curvsharp2curvsharp} is also proved.
\end{proof}

\begin{rmk} \label{rmk:curv-sharp-threshold}
Theorem \ref{thm:curvsharp2curvsharp} states that if $x$ is curvature sharp, then $x$ is also $2$-curvature sharp. The following argument shows that $N=2$ is the optimal threshold. Consider $x$ to be a curvature sharp vertex with the following two properties:
		\begin{itemize}
			\item[(i)] All transition rates $p_{yy'}$ between two different vertices $y,y' \in S_1^G(x)$ vanish,
			\item[(ii)] For every vertex $z \in S_2^G(x)$ there is at most one $y \in S_1(x)$ with $p_{xy}p_{yz} > 0$.
		\end{itemize}
		It follows from the proof of Theorem \ref{thm:curvsharpeq-Q} above that in this case, we have for $M_N(x)$ with any $N > 2$:
		\[ (M_N(x))_{yy'} = \left( \frac{1}{2} - \frac{1}{N}\right) p_{xyp_xy'} > 0, \]
		so $M_N(x)$ is diagonal dominant with non-positive diagonal entries. Thus $M_N(x)$ has at least one negative eigenvalue, which means $x$ is no longer $N$-curvature sharp for $N > 2$. Particular examples where all vertices have this property are regular trees.
\end{rmk}
	
As we shall see in the following proposition, one can weaken the conditions in the above Theorem \ref{thm:curvsharpeq-Q} by replacing the term $K_\infty^{d(x,\cdot)}(x)$ with an unknown constant. Moreover, there is another reformulation in terms of the $\Gamma_2$-operator.
\begin{prop} \label{prop:curvsharpeq_short}
Curvature sharpness of $x$ is equivalent to any of the following statements:
\begin{enumerate}
	\item $\left(Q(x) - 2\lambda \Gamma(x)_{S_1} \right) \mathbf 1_m=0$ for some $\lambda \in \mathbb R$.
	\item $Q(x) \mathbf 1_m = \lambda \mathbf p_x$ for some $\lambda \in \mathbb R$
	\item There is $\lambda \in \mathbb{R}$ such that  $\Gamma_2(d(x,\cdot),f)= \lambda \Delta f$ at $x$ for all $f \in C(V)$

\end{enumerate}
In fact, the value $\lambda$ in (a),(b),(c) need to be $\lambda = \frac{1}{2} K_\infty^{d(x,\cdot)}(x)$.
\end{prop}

\begin{proof}
    $(1)\Leftrightarrow(2)$ is straightforward.
	Next, $(2)$ is equivalent to
	\[
	\mathbf{1}_m^\top Q(x) \vec{f}_{S_1} = \lambda \mathbf p_x^\top \vec{f}_{S_1},
	\]
	for all $f \in C(V)$. By \eqref{eq:QequalsGamma2}, the above equation can be translated directly as 
	\[
	\Gamma_2(d_G(x,\cdot),f) = \lambda \Delta f.
	\]
	Thus we proved $(3)\Leftrightarrow (2)$.
	Moreover, by assuming $(3)$, we plug in $f=d(x,\cdot)$ and obtain
	\[
	\frac 1 2 D_x K_{\infty}^{d(x,\cdot)}(x) =K_{\infty}^{d(x,\cdot)}(x)\Gamma(d(x,\cdot))(x) =\lambda D_x.
	\]
	This shows that $\lambda = \frac 1 2 D_x K_{\infty}^{d(x,\cdot)}(x)$, proving the equivalence of all assertions and the curvature sharpness due to Theorem \ref{thm:curvsharpeq-Q}.
\end{proof}

We finish this subsection with a theorem providing a list of all curvature sharpness equivalences derived before.
	
\begin{thm}[Curvature sharpness equivalences] \label{thm:curvsharpeq}
	Let $(G,P)$ be a Markovian weighted graph. The following statements are equivalent for $x \in V$ with $S_1(x) = \{y_1,\dots,y_m\}$:
	\begin{enumerate}
		\item $x$ is curvature sharp, i.e., $K_N(x)=K_N^{d(x,\cdot)}(x)$ for some $N \in (0,\infty]$.
		\item There is $\lambda \in \mathbb{R}$ such that  $\Gamma_2(d(x,\cdot),f)= \lambda \Delta f$ at $x$ for all $f \in C(V)$
		\item $Q(x) \mathbf 1_m = \lambda \mathbf p_x$ for some $\lambda \in \mathbb R$
		\item $\left(Q(x) - 2\lambda \Gamma(x)_{S_1} \right) \mathbf 1_m=0$ for some $\lambda \in \mathbb R$.
		\item $x$ is $2$-curvature sharp.
		\item $x$ is curvature sharp with respect to $d_P(x,\cdot)$, and $d_P(x,\cdot)=d_G(x,\cdot)$ on $S_2^P(x)$.
	\end{enumerate}
	In all cases, $\lambda$ can be chosen to be $\frac 1 2 K_{\infty}^{d(x,\cdot)}(x)$, and $N$ can be chosen to be $2$.
\end{thm}

\begin{proof}
We refer to Proposition \ref{prop:curvsharpeq_short} for the equivalences (1)-(4), Theorem \ref{thm:curvsharp2curvsharp} for ``(1) $\Leftrightarrow$ (5)'', and Proposition \ref{prop:curvsharpsubgraph} for ``(1) $\Leftrightarrow$ (6)''.
\end{proof}

	\subsection{Geometric curvature sharpness properties}
	
	Theorem \ref{thm:curvsharpeq} states that curvature sharpness of a vertex $x$ is equivalent to the identity $Q(x){\bf{1}}_m = \frac{K_\infty^{d_G(x,\cdot)}(x)}{2} {\bf{p}}_x$ where $m = |S_1(x)|$.
	This leads to the following analytic characterisation of curvature sharpness.
	
	\begin{thm} \label{thm:curvsharpanal}
		Let $(G,P)$ be a Markovian weighted graph. A vertex $x \in V$ is curvature sharp if and only if the following identities for all $y \in S_1(x)$ are satisfied:
		\begin{equation} \label{eq:curvsharp-one-ball}
		p_{xy} \left( 4 p_{yx} + 2 \sum_{\substack{y' \in S_1(x)\\ y' \neq y}} p_{yy'} - \frac{4}{D_x} \sum_{y' \in S_1(x)} p_{xy'}p_{y'x} - \frac{1}{D_x} \sum_{y',y'' \in S_1(x)} p_{xy'}p_{y'y''} + p_{yy} \right) = \sum_{\substack{y' \in S_1(x)\\ y' \neq y}} p_{xy'}p_{y'y}. 
		\end{equation}
		
		In particular, curvature sharpness of a vertex $x$ is determined by the transition probabilities of the $1$-ball $B_1(x)$.
	\end{thm}
	
	\begin{rmk}
		We like to emphasize that the final statement in Theorem \ref{thm:curvsharpanal} is a very surprising geometric fact which follows from the explicit equations in the theorem. The definition of curvature sharpness involves the Bakry-\'Emery curvatures $K_N(x)$ which are not determined by information about $B_1(x)$. Moreover, the more restrictive property of $\infty$-curvature sharpness is also not determined by $B_1(x)$. Both concepts require information about the $2$-ball $B_2(x)$. In contrast, the property of curvature sharpness (for some dimension) can be determined once we know the structure of $B_1(x)$ in the Markovian setting.
	\end{rmk}

The derivation of these explicit identities in Theorem \ref{thm:curvsharpanal} is exactly the same as the derivation of the equations for the normalized curvature flow in Subsection \ref{subsec:der-curv-flow}, and it is based on the expressions in \eqref{eq:KN0-alt} for $K_\infty^{d_G(x,\cdot)}(x)$ and in the following lemma for $(4 Q(x){\bf{1}}_m)_i$. For that reason we leave this calculation to the readers.
\begin{lemma} \label{lem:4Q}
	Let $P \in \mathcal{M}_G$, $x \in V$ and $S_1(x) = \{y_1,\dots,y_m\}$. Then we have
	for $i \in \{1,\dots,m\}$
	\[ (4Q(x) {\bf{1}}_m)_i = p_{xy_i} \left( D_x-D_{y_i} + 4p_{y_ix} + 2 \sum_{j \neq i} p_{y_iy_j} \right) - \sum_{j \neq i} p_{xy_j} p_{y_jy_i},
	\]
	which is a homogeneous polynomial of degree $2$ in the transition probabilities.
\end{lemma}
	
	\begin{proof}
		Let $y=y_i$. We use the formulas \eqref{eq:Qepsx1} and \eqref{eq:Qepsx2} to compute $S := 4Q(x)_{yy}+\sum_{j \neq i} 4Q(x)_{yy_j}$. We first compute the contribution of the terms involving $z \in S_2(x)$ in these formulas to $S$:
		\begin{multline*} 
			3 p_{xy} \sum_{z \in S_2(x)} p_{yz} - p_{xy} \left( \sum_{z \in S_2(x)} p_{xy}p_{yz}^2\, q_z + \sum_{j \neq i} \sum_{z \in S_2(x)} p_{xy_j}p_{y_jz}p_{yz}\, q_z \right) = \\
			\sum_{\substack{z \in S_2(x) \\ p_{xy}p_{yz} \neq 0}} p_{xy}p_{yz}\left( 3 - \frac{4}{p_{xz}^{(2)}} \sum_{j=1}^m p_{xy_j}p_{y_jz} \right) = - p_{xy} \sum_{z \in S_2(x)} p_{yz}.
		\end{multline*}
		This contribution can be rewritten as
		\begin{equation} \label{eq:term14Q} 
			-p_{xy} \sum_{z \in S_2(x)} p_{yz} = p_{xy} \left( p_{yx} + \sum_{j \neq i} p_{yy_j} - D_y \right). 
		\end{equation}
		Next we compute the contribution of all the other terms in these formulas to $S$:
		\begin{multline} \label{eq:term24Q}
			2 p_{xy}^2 + 2 \sum_{j \neq i} p_{xy}p_{xy_j} + 3 p_{xy}p_{yx} - D_x p_{xy} + (3-2)p_{xy} \sum_{j \neq i} p_{yy_j} + (1-2) \sum_{j \neq i} p_{xy_j}p_{y_jy} = \\
			p_{xy} \left( 3 p_{yx} + D_x + \sum_{j \neq i} p_{yy_j} \right) - \sum_{j \neq i} p_{xy_j}p_{y_jy}. 
		\end{multline}
		The lemma follows now by adding \eqref{eq:term14Q} and \eqref{eq:term24Q}.
	\end{proof}

	In our next result, we show that volume homogeneity at a vertex $x$ implies the curvature sharpness at $x$. This result can be viewed as a generalization of \cite[Theorem 1.14]{CKLP-21} to degenerate weighted graphs.
	
	\begin{defin}[Volume homogeneity]
		A vertex $x$ is \emph{volume homogeneous} if and only if $p_y^-:=p_{yx}$ and $p_y^+:= \sum_{z\in S_2(x)} p_{yz}$ do not depend on $y\in S_1(x)$, i.e., $p_y^{\pm} = p_{y'}^{\pm}$ for all $y,y' \in S_1(x)$.
	(Since our weighting scheme is assumed to be Markovian, this also means that $p_y^0 + p_{yy}$ is independent of $y \in S_1(x)$, where $p_y^0 := \sum_{y'\in S_1(x) \setminus {y}} p_{yy'}$.)
	\end{defin}
	\begin{thm}[Volume homogeneity implies curvature sharpness]
		If $(G,P)$ is a Markovian weighted graph with a {\bf{reversible}} weighting scheme $P$ (that is, we have $\pi_x p_{xy} = \pi_y p_{yx}$ for all $x,y \in V$ with a row vector $\pi$ with all entries positive such that $\pi P = \pi$) and if $x \in V$ is volume homogeneous, then $x$ is curvature sharp.
	\end{thm}

\begin{proof}
	By the formula in Lemma \ref{lem:4Q}, we have
	\begin{align*}
	 (4Q(x) {\bf{1}}_m)_i 
	 &= p_{xy_i}\left( D_x-D_{y_i} + 4p_{y_ix} + 2 \sum_{j \neq i} p_{y_iy_j} \right) - \sum_{j \neq i} p_{xy_j} p_{y_jy_i} \\
	 &= p_{xy_i}\left( D_x-D_{y_i} + 4p_{y_i}^- + 2 \sum_{j \neq i} p_{y_iy_j} \right)- \left( p_{xy_i}^{(2)}-p_{xy_i}p_{y_iy_i}-p_{xx}p_{xy_i} \right) \\
	 &=p_{xy_i}\left(1-D_{y_i}+ 4p_{y_i}^-+ 2(1-p_{y_iy_i}-p_{y_i}^{-} - p_{y_i}^{+})+p_{y_iy_i}\right) -p_{xy_i}^{(2)} \\
	 &=p_{xy_i}\left(2+ 2p_{y_i}^{-}- 2p_{y_i}^{+} \right) -p_{xy_i}^{(2)}.
	\end{align*}
	Next, we have for any $y \in S_1(x)$
	\begin{align*}
	\pi_xp_{xy}^{(2)}
	&=\sum_z \pi_x p_{xz}p_{zy} 
	= \sum_z \pi_y p_{yz}p_{zx} \\
	&= \pi_y \left(p_{yx}p_{xx} + \sum_{y'\in S_1(x)} p_{yy'}p_{y'x} \right)
	= \pi_y p_y^{-} (p_{xx}+p_y^0+p_{yy})
	= \pi_x p_{xy} (p_{xx}+p_y^0+p_{yy}).
	\end{align*}
	Using $\pi_x\neq 0$, we conclude that 
	$$(4Q(x) {\bf{1}}_m)_i = \left(2+ 2p_{y_i}^{-}- 2p_{y_i}^{+} -(p_{xx}+p_{y_i}^0+p_{y_iy_i}) \right)p_{xy_i}=4\lambda p_{xy_i},$$
	with $\lambda$ independent of $y_i$ due to the volume-homogeneity condition. Thus, $x$ is curvature sharpness in view of the characterization $Q(x)\mathbf{1}_m=\lambda \mathbf p_x$ (see Theorem \ref{thm:curvsharpeq})).
\end{proof}

	\section{Examples of curvature sharp weighted graphs}
	\label{sec:cursharpgraphs}
	
	While the last section was concerned with individual curvature sharp vertices we investigate in this section weighted graphs for which all vertices are curvature sharp. Henceforth, all our considerations are restricted to the case of unmixed graphs $G =(V,E)$ (that is, $G$ does not have one-sided edges) and to Markovian weighting schemes $P$ without laziness, unless stated otherwise.
	
	\subsection{Curvature sharp weighting schemes for complete graphs} \label{subsec:curvsharp_complete}
	
	The determination of \emph{all} curvature sharp Markovian weighting schemes for graphs $G=(V,E)$ admitting triangles is an extremely challenging task. For example, in the case of the complete graph $K_n$, the curvature sharpness conditions \eqref{eq:curvsharp-one-ball} simplify for every pair $x \neq y$ of vertices to
	\begin{multline*}
		0 = p_{xy}\left( -4p_{yx} - 2 \sum_{z \in V \backslash \{x\}} p_{yz} + 4 \sum_{z \in V} p_{xz}p_{zx} + \sum_{z \in V} \sum_{w\in V \backslash \{x\}} p_{xz}p_{zw} \right) + \sum_{z \in V} p_{xz}p_{zy} \\
		= p_{xy} \left( -4p_{yx} - 2(1-p_{yx}) + 4 \sum_{z \in V} p_{xz}p_{zx} + \sum_{z \in V} p_{xz} (1-p_{zx}) \right) + \sum_{z \in V} p_{xz}p_{zy} \\
		= p_{xy} \left( -2 -2 p_{yx} + 3 \sum_{z\in V} p_{xz}p_{zx} + \sum_{z \in V} p_{xz} \right) + \sum_{z \in V} p_{xz}p_{zy} \\
		= p_{xy}(-1-2p_{yx}) + 3 p_{xy} \sum_{z\in V} p_{xz} p_{zx} + \sum_{z \in V} p_{xz} p_{zy}.
	\end{multline*}
	This is equivalent to
	\begin{equation} \label{eq:curv-sharp-compl-graph} 
		p_{xy} ( 1+ 2 p_{yx} ) = 3 p_{xy} \sum_{z \in V} p_{xz} p_{zx} + \sum_{z \in V} p_{xz} p_{zy} = 3 p_{xy} p_{xx}^{(2)} + p_{xy}^{(2)}. 
	\end{equation}
	It is easy to see that, on $K_n$, the simple random walk $p_{xy} = \frac{1}{n-1}$ for all $x \neq y$ is always curvature sharp. We assume that this is the only \emph{non-degenerate} curvature sharp Markovian weighting scheme without laziness on $K_n$, but we are currently only able to prove this for $K_3$, as stated in Proposition \ref{prop:K3} in the Introduction. Let us now provide the proof of that proposition.
	
	\begin{proof}[Proof of Proposition \ref{prop:K3}]
		A curvature sharp weighting scheme on $K_3$ with vertices $\{0,1,2\}$ needs to satisfy simultaneously the following $6$ polynomial equations:
		\begin{eqnarray*}
			p_{01}(1+2p_{10}) &=& 3p_{01}(p_{01}p_{10}+p_{02}p_{20}) + p_{02}p_{21}, \\
			p_{02}(1+2p_{20}) &=& 3p_{02}(p_{01}p_{10}+p_{02}p_{20}) + p_{01}p_{12}, \\
			p_{10}(1+2p_{01}) &=& 3p_{10}(p_{10}p_{01}+p_{12}p_{21}) + p_{12}p_{20}, \\
			p_{12}(1+2p_{21}) &=& 3p_{12}(p_{10}p_{01}+p_{12}p_{21}) + p_{10}p_{02}, \\
			p_{20}(1+2p_{02}) &=& 3_{p20}(p_{20}p_{02}+p_{21}p_{12}) + p_{21}p_{10}, \\
			p_{21}(1+2p_{12}) &=& 3_{p21}(p_{20}p_{02}+p_{21}p_{12}) + p_{20}p_{01}
		\end{eqnarray*}
		together with the Markovian properties 
		\begin{eqnarray*}
			p_{01} + p_{02} &=& 1, \\
			p_{10} + p_{12} &=& 1, \\
			p_{20} + p_{21} &=& 1.
		\end{eqnarray*}
		The solution set $S$ is a real affine algebraic variety, and we need to intersect this algebraic variety with the cube 
		$$ Q := \{ (p_{01},p_{02},p_{10},p_{12},p_{20},p_{21}) \in [0,1]^6 \} $$
		to find all curvature sharp weighting schemes. Maple provides us with the following solution set:
		\begin{multline*} 
			S = \left\{ \left(\frac{1}{2},\frac{1}{2},\frac{1}{2},\frac{1}{2},\frac{1}{2},\frac{1}{2}\right), \left(\frac{4}{3},-\frac{1}{3},\frac{1}{2},\frac{1}{2},-\frac{1}{3},\frac{4}{3}\right), \left(\frac{1}{2},\frac{1}{2},\frac{4}{3},-\frac{1}{3},\frac{4}{3},-\frac{1}{3}\right)\right. \\
			\left. \left(-\frac{1}{3},\frac{4}{3},-\frac{1}{3},\frac{4}{3},\frac{1}{2},\frac{1}{2}\right)\right\} \cup \left\{ \left( x,1-x,\frac{2x-1}{3x-2},\frac{x-1}{3x-2},\frac{2x-1}{3x-1},\frac{x}{3x-1}\right) \mid x \in \mathbb {R} \right\}. 
		\end{multline*}
		It is easy to see that 
		$$ S \cap Q = \left\{ \left(\frac{1}{2},\frac{1}{2},\frac{1}{2},\frac{1}{2},\frac{1}{2},\frac{1}{2}\right), \left(0,1,\frac{1}{2},\frac{1}{2},1,0 \right), \left(\frac{1}{2},\frac{1}{2},0,1,0,1\right),\left(1,0,1,0,\frac{1}{2},\frac{1}{2} \right) \right\}. $$
		In conclusion $K_3$ has precisely one non-degenerate curvature sharp weighting scheme (the simple random walk) and three degenerate curvature sharp weighting schemes. 
	\end{proof}
	
	\smallskip
	
	The arguments in this proof are restricted to $K_3$, since Maple is no longer able to solve the corresponding polynomial equations in the case $n=4$.
	We have, however, for general $n$, the following degenerate curvature sharp Markovian weighting schemes on $K_n$.
	
	\smallskip
	
	Let $K_n$ be the complete graph with vertices $\{0,1,\dots,n-1\}$ and let $1 < m < n$. Then the induced subgraph of the vertices $\{0,1,\dots,m-1\}$ is also complete graph which we denote by $K_m$, and we can choose the simple random walk on $K_m$ and extend it to a curvature sharp weighting scheme on $K_n$ as follows: We set $p_{ji} = 0$ for any pair $i,j \in \{m,\dots,n-1\}$ and
	$p_{ji} = \frac{1}{m}$ for all $i \in \{0,\dots,m-1\}$ and $j \in \{m,\dots,n-1\}$.
	It is easily checked that \eqref{eq:curv-sharp-compl-graph} holds in this case for any choice of $x$ and  $y$. 
	For example, if $x \in K_n \backslash K_m$ and $y \in K_m$, then we have $p_{xy} = \frac{1}{m}$, $p_{xz}p_{zx} =0$ for all $z$, and $p_{xz}p_{zy} = \frac{1}{(m-1)m}$ only for $z \in K_m \backslash \{y\}$, and both sides of \eqref{eq:curv-sharp-compl-graph} are equal to $\frac{1}{m}$.
	
	\subsection{Curvature sharp weighting schemes for arbitrary connected graphs}
	\label{subsec:curv-sh-tr-fr-edge}
	
	It is a natural question whether any finite connected simple graph $G=(V,E)$ admits a curvature sharp Markovian weighting scheme which is weakly connected. By weakly connected we mean that there is a path between any two vertices in the underlying undirected graph, which is obtained by ignoring the directions of the edges. Theorem \ref{thm:weakconncurvsharp} from the Introduction provides a positive answer to this question. The curvature sharp weighting scheme described there is based on the distance structure of spheres around a clique $G_0 = (E_0,V_0) \cong K_n$ with $n \ge 2$. In plain words, this weighting scheme is the simple random walk on $K_n$ and there is, for any vertex $x \in V \backslash V_0$, a unique vertex $x_0 \in V$ ($x_0$ can coincide with $x$ if $d_G(x,V_0)=1$) adjacent to $K_n$ and a unique directed path from $x$ to $x_0$ with strictly decreasing distance to $V_0$ and transition probabilities $=1$ along all of its edges. Moreover, the vertex $x_0$ has equal transition probabilities to all its neighbours in $K_n$ and vanishing transition probabilities to all other neighbours.  
	
	\begin{proof}[Proof of Theorem \ref{thm:weakconncurvsharp}]
		
		Let us first prove curvature sharpness for any vertex $x \in V$ with $d(x,K_n) \ge 2$.
		There is precisely one edge $\{x,y_0\} \in E$ for which $p_{xy_0}$ is non-zero, and $y_0 \in V$ must satisfy $d(y_0,K_n) = d(x,K_n)-1$. For this vertex $y_0$ we have
		$p_{xy_0} = 1$ and $p_{y_0x}=0$. Note also that we have $p_{y_0y} = 0$ for any neighbour $y \in S_1(x)$. Plugging this information into \eqref{eq:curvsharp-one-ball}
		yields
		\begin{multline*}
			p_{xy_0}\left( -4p_{y_0x} - 2 \sum_{\substack{y \in S_1(x) \\ y \neq y_0}} p_{y_0y} + 4 \sum_{y \in S_1(x)} p_{xy}p_{yx} + \sum_{y,y' \in S_1(x)} p_{xy}p_{yy'} \right) + \sum_{y \neq y_0} p_{xy}p_{yy_0} \\ = \left( -4 \cdot 0 - 2 \cdot 0 + 4 p_{xy_0}p_{y_0x} + p_{xy_0}\sum_{y' \in S_1(x)} p_{y_0y'} \right) + 0 = 0
		\end{multline*}
		and, for $y \in S_1(x) \backslash \{y_0\}$,
		\begin{multline*}
			p_{xy}\left( -4p_{yx} - 2 \sum_{\substack{y' \in S_1(x) \\ y' \neq y }} p_{yy'} + 4 \sum_{y' \in S_1(x)} p_{xy'}p_{y'x} + \sum_{y',y'' \in S_1(x)} p_{xy'}p_{y'y''} \right) + \sum_{y' \neq y} p_{xy'}p_{y'y} \\ = 0 + \sum_{y' \neq y} p_{xy'}p_{y'y} = p_{xy_0}p_{y_0y} = 0.  
		\end{multline*}
		This confirms curvature sharpness of all vertices in $S_r(K_n)$ for all $r \ge 2$. 
		
		\smallskip
		
		Let us next consider a vertex $x \in V$ with $d(x,K_n)=1$: For any neighbour $y \sim x$ which is not in $K_n$, we have
		\begin{multline*}
			p_{xy}\left( -4p_{yx} - 2 \sum_{\substack{y' \in S_1(x) \\ y' \neq y}} p_{yy'} + 4 \sum_{y' \in S_1(x)} p_{xy'}p_{y'x} + \sum_{y',y'' \in S_1(x)} p_{xy'}p_{y'y''} \right) + \sum_{y' \neq y} p_{xy'}p_{y'y} \\ = 0 + \sum_{y' \in S_1(x) \cap K_n} p_{xy'} \underbrace{p_{y'y}}_{=0} = 0.
		\end{multline*}
		Assuming, $x$ has $k$ neighbours in $K_n$, we obtain for $y \sim x$ with $y \in K_n$:
		\begin{multline*}
			p_{xy}\left( -4p_{yx} - 2 \sum_{\substack{y' \in S_1(x) \\ y' \neq y}} p_{yy'} + 4 \sum_{y' \in S_1(x)} p_{xy'}p_{y'x} + \sum_{y',y'' \in S_1(x)} p_{xy'}p_{yy''} \right) + \sum_{y' \neq y} p_{xy'}p_{y'y} \\ = \frac{1}{k}\left( -4 \cdot 0 - 2 
			\cdot \frac{k-1}{n-1} + 4 \cdot 0 + k \cdot \left( \frac{1}{k} \cdot \frac{k-1}{n-1} \right) 
			\right)  + \frac{k-1}{k} \frac{1}{n-1} = 0.
		\end{multline*}
		This shows curvature sharpness of all vertices in $S_1(K_n)$.
		
		\smallskip
		
		Finally, let us consider a vertex $x \in K_n$: For any neighbour $y \sim x$ which is not in $K_n$, we have
		\begin{multline*}
			p_{xy}\left( -4p_{yx} - 2 \sum_{\substack{y' \in S_1(x) \\ y' \neq y}} p_{yy'} + 4 \sum_{y' \in S_1(x)} p_{xy'}p_{y'x} + \sum_{y',y'' \in S_1(x)} p_{xy'}p_{y'y''} \right) + \sum_{y' \neq y} p_{xy'}p_{y'y} \\ = 0 + \sum_{y' \in S_1(x) \cap K_n} p_{xy'} \underbrace{p_{y'y}}_{=0} = 0.
		\end{multline*}
		For any neighbour $y \in S_1(x) \cap K_n$, we have
		\begin{multline*}
			p_{xy}\left( -4p_{yx} - 2 \sum_{\substack{y' \in S_1(x) \\ y' \neq y }} p_{yy'} + 4 \sum_{y' \in S_1(x)} p_{xy'}p_{y'x} + \sum_{y',y'' \in S_1(x)} p_{xy'}p_{y'y''} \right) + \sum_{y' \neq y} p_{xy'}p_{y'y} \\ 
			= \frac{1}{n-1}\left( - 4 \cdot \frac{1}{n-1} - 2 \cdot \frac{n-2}{n-1} + 4 \cdot \frac{n-1}{(n-1)^2} + \frac{(n-1)(n-2)}{(n-1)^2} \right) + \frac{n-2}{(n-1)^2} = 0.
		\end{multline*}
		This shows curvature sharpness of all vertices in $K_n$.
		
		\smallskip
		
		Weakly connectedness of this curvature sharp weighting scheme follows straightforwardly from the fact that there is a directed path from any vertex $x \in V \backslash K_n$ to $K_n$ of length $d(x,K_n)$.  
	\end{proof}
	
	The weighting scheme in Theorem \ref{thm:weakconncurvsharp}
	has some transition rates which are not in $\{0,1\}$, for example, the non-zero transition rates for vertices of $S_1(K_n)$ which have more than one neighbour in $K_n$ and, if $n \ge 3$, all transition rates between vertices of $K_n$. It is an interesting question whether all graphs have weakly connected curvature sharp weighting schemes with only $\{0,1\}$-transition rates. It turns out that this is true for every graph $G = (V,E)$ which has at least one edge $e = \{v,w\}$ which is not contained in a triangle. In this case, the following weighting scheme is weakly connected and curvature sharp: Let $T = (V,E')$ be a spanning tree of $G$ with $e \in E'$
	and $d_T(x,v) = d_G(x,v)$ for all vertices $x \in V$. It is easy to see that such a spanning tree exists. Then we define $p_{vw} = p_{wv} = 1$ and, for all $x \neq v,w$, $p_{x y} = 1$ if and only if $x$ and $y$ are adjacent in $T$ and if $d_T(x,v) = d_T(y,v)+1$. On the other hand, there exist unmixed graphs which do not admit curvature sharp Markovian weighting schemes without laziness all of whose transition rates are in $\{0,1\}$. The smallest counterexample is the complete graph $K_3$. This follows from Proposition \ref{prop:K3}, since each weighting scheme there has a directed edge $(x,y)$ with $p_{x y} = \frac{1}{2}$.
	
	
	\subsection{Curvature sharp weighting schemes for triangle-free graphs}
	\label{subsec:curv-sh-tr-free}
	
    It is straightforward to see that the curvature sharpness condition \eqref{eq:curvsharp-one-ball} at a vertex $x \in V$ for Markovian weighted graphs $(G,P)$ without laziness reduces to the following much simpler condition in the case that $x$ is not contained in any triangle:
	\begin{equation} \label{eq:curvsharp-tr-free}
		0 = p_{xy_i} \left( \sum_{j=1}^m p_{xy_j}p_{y_jx} - p_{y_i x} \right)
	\end{equation}
	for all $y_i \in S_1(x)$. Before we give the proof of Theorem \ref{thm:curv-sh-tr-free} from the Introduction characterising all non-degenerate curvature sharp Markovian weighting scheme for triangle-free graphs, we first prove the following useful lemma.
	
	\begin{lemma} \label{lem:tr-free-curv-sh}
	    Let $(G,P)$ be a non-degenerate curvature sharp Markovian weighted graph without laziness and 
		$x \in V$ be a vertex not contained in any triangle.  Then all transition probabilities $p_{yx}$
		agree for all neighbours $y \in S_1(x)$ of $x$.
	\end{lemma}
	
	\begin{proof}
		Non-degeneracy guarantees that we have $p_{xy_i} > 0$ for all $y_j \sim x$, and it follows from \eqref{eq:curvsharp-tr-free} that
		we have
		$$ p_{y_i x} = \sum_{j=1}^m p_{xy_j}p_{y_jx}. $$
		Since the right hand side is independent of $i$, we have
		$p_{y_i x} = p_{y_j x}$ for any pair $y_i, y_j \in S_1(x)$.
	\end{proof}
	
	The observation in Lemma \ref{lem:tr-free-curv-sh} is crucial for the proof of Theorem \ref{thm:curv-sh-tr-free}:
	
	\begin{proof}[Proof of Theorem \ref{thm:curv-sh-tr-free}]
		Let $P$ be a non-degenerate curvature sharp Markovian weighting scheme without laziness. Let $x \in V$. Since $p_{yx}$ is independent of $y \in S_1(x)$, by Lemma \ref{lem:tr-free-curv-sh}, we can define $c_x = p_{y x}$ for any choice of $y \in S_1(x)$. On the other hand, the Markovian property needs to be satisfied, that is
		\begin{equation} \label{eq:cysum} 
			\sum_{y \sim x} c_y = \sum_{y \sim x} p_{xy} = 1 \quad \text{for all $x \in V$.}
		\end{equation}
		This property can be rewritten with the help of the adjacency matrix as 
		\begin{equation} \label{eq:Ag-cond} 
			A_G {\bf{c}} = {\bf{1}}_M 
		\end{equation} 
		with $M = |S_1(x)|$. Conversely, any vector $\bf{c}$ satisfying \eqref{eq:Ag-cond} gives rise to such a weighting scheme by defining $p_{xy} = c_x$ for all $y \in S_1(x)$.  
		
		The inhomogeneous equation \eqref{eq:Ag-cond} may not have any solution in $(0,1]^M$. However, if it has at least one solution ${\bf{c}}_0 \in (0,1]^M$, then
		this solution is unique if $A_G$ is invertible since then ${\bf{c}}_0 = A_G^{-1} {\bf{1}}_M$. If $A_G$ is not invertible, then ${\bf{c}}_0$ cannot be the only solution in $(0,1]^M$, since all linear equations in \eqref{eq:Ag-cond} involving variables $c_x$ with $(c_0)_x = 1$ (with non-zero coefficients) must be trivial, that is, any other solution $\bf{c} \in (0,1]^M$ of \eqref{eq:Ag-cond} is forced to have also $c_x = 1$, and for all other variables $c_x$ we have $(c_0)_x \in (0,1)$, and these parameters can be perturbed along the kernel of $A_G$. Moreover, these other variables must exist unless $G = K_2$, in which case $A_G$ is invertible.
		
		Finally, convexity of the solution set
		follows directly from convexity of the solution set of $A_G {\bf{c}} = {\bf{1}}_M$
		in $\mathbb{R}^M$ and the convexity of $(0,1]^M$.
	\end{proof}
	
	Theorem \ref{thm:curv-sh-tr-free} is very useful to find all non-degenerate curvature sharp weighting schemes for various triangle free combinatorial graphs. In the particular case of a bipartite graph with vertex partition $V = V_0 \cup V_1$, the curvature sharpness conditions \eqref{eq:cysum} can be separated into two independent systems of inhomogeneous linear equations, one such system for each vertex set $V_i$. Let us now continue with the proof of Corollary
	\ref{cor:bip-uniq-curv-sharp}.
	
	\begin{proof}[Proof of Corollary \ref{cor:bip-uniq-curv-sharp}]
		Since $G$ is bipartite, the spectrum of $A_G$ is symmetric and, in the case of an odd number of vertices, $0$ must be an eigenvalue of $A_G$. Then $A_G$ is not invertible and $G$ cannot have a unique non-degenerate curvature sharp weighting scheme.
	\end{proof}
	
	It follows also easily from Theorem \ref{thm:curv-sh-tr-free}
	that there are bipartite graphs which do not admit non-degenerate curvature sharp weighting schemes. A simple example is $G=(V,E)$ with $V = \{0,1,2,3\}$ and $E = \{ \{0,2\}, \{1,2\}, \{1,3\} \}$. It follows from the Markovian property that $p_{02} = c_2 = 1$ and $p_{12} + p_{13} = c_2+c_3 = 1$. This implies $c_3 = 0$, in contradiction to the non-degeneracy condition. Moreover, the statement in Corollary \ref{cor:bip-uniq-curv-sharp} is not an "if and only if", as the following example shows:
	
	\begin{ex}[Complete bipartite graph $K_{2,2}$]
		In the case of this graph we have many non-degenerate curvature sharp weigthing schemes. If we enumerate the vertices in such a way that $0,1$ are on the left hand side of this graph and $2,3$ are on the right hand side, and if we denote the variable corresponding to $j$ by $c_j$, then the only conditions we obtain are $c_0+c_1=1$ and $c_2+c_3=1$, which leads to many solutions. 
	\end{ex}
	
	Another simple example having many non-degenerate curvature sharp weighting schemes is the star graph.
	
	\begin{ex}[Star graph]\label{ex:stargraph}
		Let $G =(V,E)$ be a star graph with centre $x \in V$, that is
		$V = \{x,y_1,\dots,y_m\}$ with $x \sim y_i$ for all $i$ and
		there are no edges between any two vertices $y_i$ and $y_j$.
		Then any choice $\sum_{y \sim x} p_{xy} = 1$ gives rise to a non-degenerate curvature sharp weighting scheme $P$ by setting $p_{y_i x} = 1$ for all $i$. 
	\end{ex}
	
	While Proposition \ref{prop:leaf-case}
	is not restricted to the case of triangle free graphs, we think that here is the right place to present its proof. Recall that this proposition states that any connected graph with a leaf must be a star graph if it admits a non-degenerate curvature sharp weighting scheme. 
	
	\begin{proof}[Proof of Proposition \ref{prop:leaf-case}]
		Let $y \in V$ be a leaf of $G$ and $P$ a non-degenerate curvature sharp weighting scheme. Let $x \in V$ be the unique neighbour of $y$. Then we must have $p_{yx} = 1$ and, by Lemma \ref{lem:tr-free-curv-sh}, we must also have $p_{zx} = 1$ for all neighbours $z \in V$ of $x$. These neighbours must again be leaves of $G$ since, in the case of $d_z \ge 2$, we would have $p_{zw} = 0$ for any other neighbour $w \in S_1(z) \backslash \{x\}$, by the Markovian property. This would contradict the non-degeneracy condition.
	\end{proof}
	
	Note that the set of all star graphs includes the complete graph $K_2$ and the path of length $2$. 
	
	Let us finally provide the proof of Corollary \ref{cor-hypcub}. 
	
	\begin{proof}[Proof of Corollary \ref{cor-hypcub}]
		The $k$-dimensional hypercube is the $k$-fold Cartesian product of $K_2$, that is, $Q^k = (K_2)^k$ and is triangle-free.
		Moreover, it is easy to verify that the hypercube is regular and $S_1$-out regular and, therefore, the simple random walk (without laziness) is a non-degenerate curvature sharp weighting scheme. Let us investigate its uniqueness.
		
		The spectrum of the adjacency matrix of $K_2$ is given by $\sigma = \{-1,1\}$. The spectrum of $Q^k$ consists then all sums $\sum_{j=1}^k a_j$ with $a_j \in \sigma$. Consequently, $0$ is in the spectrum of $Q^k$ if and only if $k$ is even. In this case the adjacency matrix is not invertible and the non-degenerate curvature sharp weighting scheme is not unique. If $k$ is odd, $0$ is not in the spectrum of $Q^k$, its adjacency matrix is invertible and the simple random walk is the unique non-degenerate curvature sharp weighting scheme.
	\end{proof}
	
	\section{Semicontinuity of curvature as function of the weighting scheme}
	\label{sec:curvcont}
	
	The curvature flow in Definition \ref{def:curvflow-Q} provides a matrix-valued function $P(t)$ with $P(0) = P_0$, which depends continously on the time parameter $t$. Therefore, it is natural to ask whether the Bakry-\'Emery curvatures of the vertices depend also continuously on the weighting schemes $P(t)$. We will see that this is only true if we consider convergence in a specific subspace $\mathcal{M}_P$ preserving vanishing and non-vanishing transition rates. In general, we have only upper semicontinuity and the curvature can jump upwards if certain transition probabilities converge to zero. For example, this is relevant in the case that we have a convergent curvature flow $P^\infty = \lim_{t \to \infty} P(t)$ with $P(0) = P_0$, since $P^\infty$ is often no longer in $\mathcal{M}_{P_0}$.
	
	Let us first introduce the subspace $\mathcal{M}_P \subset \mathcal{M}_G$. As before, let $G=(V,E)$ be a mixed combinatorial graph and $P \in \mathcal{M}_G$
	be an associated weighting scheme. The subspace $\mathcal{M}_P \subset \mathcal{M}_G$ is the set of all stochastic matrices with the same pattern of non-zero transition probabilities as $P$:
	$$ \mathcal{M}_P := \{ P' \in \mathcal{M}_G:\, p'_{xy} > 0 \, \Longleftrightarrow\, p_{xy} > 0 \,\, \forall\, x,y \in V \}. $$  
	Therefore, we have for any $P' \in \mathcal{M}_P$ that the mixed subgraphs of $G$ corresponding to $P$ and $P'$ are equal: $G_{P'} = G_P$.
	
	\begin{thm}[Curvature semicontinuity] \label{thm:curvsemicont}
		Let $P_k \in \mathcal{M}_G$ be a sequence converging to $P \in \mathcal{M}_G$. Then we have for any vertex $x \in V$:
		\begin{equation} \label{eq:curvuppcont} 
			\limsup_{k \to \infty} K_{P_k,N}(x) \le K_{P,N}(x). 
		\end{equation}
		If $P_k \in \mathcal{M}_P$ for all $k$ and $P \in \mathcal{M}_P$, the sequence $K_{P_k,N}(x)$ converges and we have
		\begin{equation} \label{eq:curvcont}
			\lim_{k \to \infty} K_{P_k,N}(x) = K_{P,N}(x). 
		\end{equation}
	\end{thm}
	
	\begin{proof}
		Let $\lim_{k \to \infty} P_k = P$ and $x \in V$. By Proposition \ref{cor:optfunc}, there exists a function $f: V \to \mathbb{R}$ with $f(x)=0$ and $\Gamma(f)(x) \neq 0$
		such that 
		$$ K_{P,N}^f(x) := \frac{1}{\Gamma^P(f)(x)}\left( \Gamma_2^P(f)(x) - \frac{1}{N} (\Delta_P f(x))^2 \right) = K_{P,N}(x). $$
		We use the notation $\Gamma^P$ and $\Gamma_2^P$ to express the dependence on the weighting scheme $P$. Since $P_k \to P$, we have $\Gamma^{P_k}(f)(x) \neq 0$ for large enough $k$ and
		$$ K_{P_k,N}(x) \le K_{P_k,N}^f(x) = \frac{1}{\Gamma^{P_k}(f)(x)}\left( \Gamma_2^{P_k}(f)(x) - \frac{1}{N} (\Delta_{P_k} f(x))^2 \right) \to K_{P,N}^f(x)
		= K_{P,N}(x). $$
		This proves \eqref{eq:curvuppcont}.
		
		Now we assume $P_k \to P \in \mathcal{M}_P$ and, additionally, $P_k \in \mathcal{M}_P$. This implies that there exists a positive constant $C > 0$ such that, for all $y \in S_1^{P_k}(x) = S_1^P(x)$, the entries $p_{xy}^k$ of $P_k$ satisfy
		$p_{xy}^k \ge C^2$. By Proposition \ref{cor:optfunc}, there exist functions $f_k: V \to \mathcal{R}$ with
		$f_k(x) =0$, $2 \Gamma^{P_k}(f_k)(x) =1$, $K_{P_k,N}(x) = K_{P_k,N}^{f_k}(x)$ and
		$$ \Vert f_k \Vert_\infty \le \frac{2}{\min\{ \sqrt{p_{xy}^k}:\, d_{P_k}(x,y) = 1 \}} \le \frac{2}{C}. $$
		By a compactness argument, there exists a convergent subsequence $f_{k_j} \to f$ with $f(x)=0$, $\Gamma^P(f)(x)=2$, $\Vert f \Vert_\infty \le \frac{2}{C}$
		and
		$$ \liminf_{k \to \infty} K_{P_k,N}^{f_k}(x) = \lim_{j \to \infty} K_{P_{k_j},N}^{f_{k_j}}(x) = K_{P,N}^f(x). $$
		Together with (a), this implies that we have
		$$ K_{P,N}(x) \le K_{P,N}^f(x) = \liminf_{k \to \infty} K_{P_k,N}^{f_k}(x) = \liminf_{k \to \infty} K_{P_k,N}(x) \le \limsup_{k \to \infty} K_{P_k,N}(x) \le K_{P,N}(x). $$
		This shows $\liminf_{k \to \infty} K_{P_k,N}(x) = \limsup_{k \to \infty} K_{P_k,N}(x) = K_{P,N}(x)$, completing the proof of \eqref{eq:curvcont}.
	\end{proof}
	
	Let us illustrate this curvature dependence on the weighting schemes in two small examples (with no one-sided edges and without laziness). We use the curvature matrices $A_\infty(x)$ for the explicit curvature computations.
	
	\begin{ex} \label{ex:curvjump1}
		Let $G = (V,E) = K_2 \times K_2$ be the square, that is $V = \{ v_0,v_1,v_2,v_3 \}$, with horizontal edges $\{v_0,v_1\},\{v_3,v_2\}$
		and vertical edges $\{v_0,v_3\},\{v_1,v_2\}$ and, for $p \in [0,1]$, 
		$$ P_p = \begin{pmatrix} 0 & p & 0 & 1-p \\ p & 0 & 1-p & 0 \\ 0 & 1-p & 0 & p \\ 1-p & 0 & p & 0\end{pmatrix} \in \mathcal{M}_G. $$
		By symmetry, the curvatures $K_\infty(v_j)$ at all vertices of $(G,P_p)$ agree. $(G,P_p)$ is non-degenerate for $p \in (0,1)$. Using \cite[(A.11)-(A.13)]{CKLP-21}, the curvature matrix $A_{P_p,\infty}(v_0)$ assumes in this case the form
		$$ A_{P_p,\infty}\infty(v_0) = \begin{pmatrix} 2p & 0 \\ 0 & 2(1-p) \end{pmatrix}. $$
		The curvature function is discontinous at $p=0$ and $p=1$ and is given by
		$$ K_{P_p,\infty}(v_0) = \begin{cases} 2\min\{p,1-p\} & \text{if $p \in (0,1)$,} \\ 2 & \text{if $p =0,1$.} \end{cases} $$
		The upper curvature bound $K_{P_p,\infty}^0(v_0)$ is given by
		$$ K_{P_p,\infty}^0(v_0) = 2 (1-2p(1-p)), $$
		and $(G,P_p)$ is $\infty$-curvature sharp for $p=0,\frac{1}{2},1$. The situation is illustrated in Figure \ref{fig:square}.
		\begin{figure}[h]
			\centering
			\includegraphics[height=8cm]{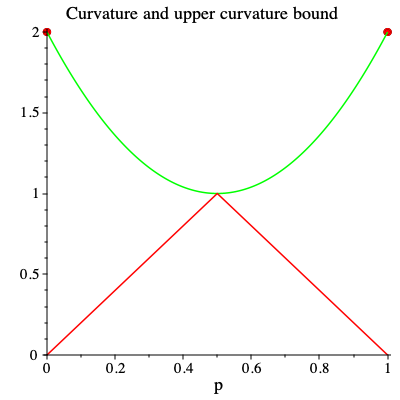}			\caption{Curvature $K_{P_p,\infty}(v_0)$ (red) and upper curvature bound $K_{P_p,\infty}^0(v_0)$ (green) of the square $G = K_2 \times K_2$ with transition probabilities $p$ along horizontal edges and $1-p$ along vertical edges. $v_0$ is $\infty$-curvature sharp for $p=0,\frac{1}{2},1$.}
			\label{fig:square}
		\end{figure}
	\end{ex}
	
	\begin{ex} \label{ex:curvjump2}
		Let $G=(V,E)$ be the path of length $3$ with vertices $V = \{v_0,v_1,v_2,v_3\}$, that is $v_j$ and $v_{j+1}$ are adjacent for $j \in \{0,1,2\}$. For $p \in [0,1]$, let $P_p$ denote the weighting scheme
		$$ P_p = \begin{pmatrix} 0 & 1 & 0 & 0 \\ 1-p & 0 & p & 0 \\ 0 & p & 0 & 1-p \\ 0 & 0 & 0 & 1 \end{pmatrix}. $$
		In the case $p \in (0,1)$, the weighted graph $(G,P_p)$ is non-degenerate and the curvature matrix at $v_1$ is given by
		$$ A_{P_p,\infty}(v_1) = \begin{pmatrix} 3p-1 & \sqrt{p(1-p)} \\ \sqrt{p(1-p)} & 2-p \end{pmatrix} $$
		and
		$$ K_{P_p,\infty}(v_1) = \lambda_{\rm{min}}(A_{P_p,\infty}(v_1)) = \frac{1}{2} + p - \frac{\sqrt{12p^2-20p+9}}{2}. $$
		As $p \to 0$, the transition rates along the edge $\{ v_1,v_2 \}$
		shrink to zero and we have
		$$ \lim_{p \to 0} K_{P_p,\infty}(v_1) = - 1. $$
		As $p \to 1$, the transition rates along the edge $\{ v_0,v_1 \}$ shrink to zero and we have
		$$ \lim_{p \to 1} K_{P_p,\infty}(v_1) = 1. $$
		On the other hand, for $p=0,1$ we have $K_{P_0,\infty}(v_1) = K_{P_1,\infty}(v_1) = 2$. This means that the curvature, as a function of $p$, is discontinuous at $p=0$ and $p=1$. Moreover, we have
		$$ K_{P_p,\infty}^0(v_1) = 2 - \frac{7p(1-p)}{2}. $$
		The vertex $v_1$ is $\infty$-curvature sharp for $p=0,1$. The situation is illustrated in Figure \ref{fig:3path}.
		\begin{figure}[h]
			\centering
			\includegraphics[height=8cm]{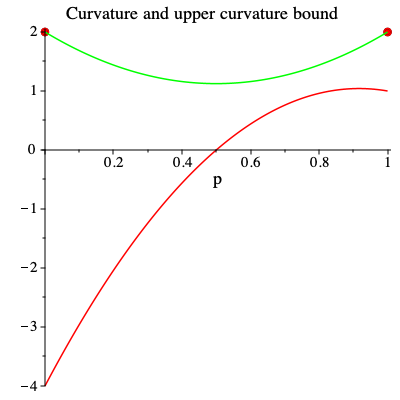}
			\caption{Curvature $K_{P_p,\infty}(v_1)$ (red) and upper curvature bound $K_{P_p,\infty}^0(v_1)$ (green) of the path of length $3$ with vertices $\{v_0,v_1,v_2,v_3\}$ and transition probabilities $p$ along the inner edge $\{v_1,v_2\}$. $v_1$ is $\infty$-curvature sharp for $p=0,1$.}
			\label{fig:3path}
		\end{figure}
	\end{ex}

	\section{Fundamental properties of the curvature flow}
	\label{sec:prop-curv-flow}
	
	This final section is devoted to the curvature flow. We derive the flow equations given in Definition \ref{def:curvflow-Q} from the motivating flow equations \eqref{eq:curv-flow0}, \eqref{eq:curv-flow0laz} and \eqref{eq:curv-flow} given in the special case of \emph{non-degenerate} weighted graphs and prove some fundamental properties of this flow.
	
	\subsection{Derivation of the curvature flow equations} \label{subsec:der-curv-flow}
	Recall from \eqref{eq:curv-flow0}, \eqref{eq:curv-flow0laz} and \eqref{eq:curv-flow} in the Introduction that the curvature flow equations
	have, in the case of non-degenerate vertices $x \in V$, the general form
	\begin{equation} \label{eq:curv-flow-gen-A} 
		{\bf{v}}_0'(x,t) = - A_{P(t),\infty}(x) {\bf{v}}_0(x,t) + C_x(t) {\bf{v}}_0(x,t), \quad p_{xx}'(t) = 0.
	\end{equation}
	The choices $C_x(t) = 0$ and $C_x(t) = \frac{2D_x}{N}$ lead to our originally considered curvature flows for dimensions $\infty$ and $N$, respectively, and the choice $C_x(t) = K_{P(t),\infty}^0(x) = K_{P(t),\infty}^{d_P(x,\cdot)}(x)$ (by Proposition \ref{prop:coinc-curv-bounds}) leads to a curvature flow preserving the Markovian property which is the focus of this paper.
	
	Left multiplication of the first part of the flow equation \eqref{eq:curv-flow-gen-A} by $2 {\rm{diag}}({\bf{v}}_0(x,t))$ yields
	\begin{equation*} 
		2 {\rm{diag}}({\bf{v}}_0(x,t)) {\bf{v}}_0'(x,t) = - 4 Q_x(t) {\bf 1}_m + 2 C_x(t) {\rm{diag}}({\bf{v}}_0(x,t)) {\bf{v}}_0(x,t),
	\end{equation*}
	where $Q_x(t)=Q(x)$ is the (slightly modified) Schur complement of $\Gamma_2(x)$ as defined in \eqref{eq:Qeps} (and we recall that $Q(x)$ and $\Gamma_2(x)$ are defined via $P(t)$ and hence time-dependent). Note that we have
	$$ {\rm{diag}}({\bf{v}}_0(x,t)) {\bf{v}}_0(x,t) = 
	(p_{xy_1}(t), p_{xy_2}(t), \cdots, p_{xy_m}(t))^\top =: {\bf{p}}_x(t) $$
	and
	$$ 2 {\rm{diag}}({\bf{v}}_0(x,t)) {\bf{v}}_0'(x,t) = (p_{xy_1}'(t), p_{xy_2}'(t), \cdots, p_{xy_m}'(t))^\top = {\bf{p}}_x'(t). $$
	Therefore, the system of differential equations \eqref{eq:curv-flow-gen-A} for all vertices $x \in V$ can be rewritten as
	\begin{equation} \label{eq:curv-flow-gen-Q} 
		{\bf{p}}_x'(t) = - 4Q_x(t) {\bf 1}_m + 2 C_x(t) {\bf{p}}_x(t), \quad p_{xx}'(t) = 0,
	\end{equation}
	where $Q_x(t) {\bf 1}_m$ is a homogeneous polynomial of degree $2$ in the transition probabilities of the $1$-ball of $B_1(x)$ by Lemma \ref{lem:4Q}.
	
	\medskip
	
	An essential advantage of the curvature flow equations \eqref{eq:curv-flow-gen-Q} compared to \eqref{eq:curv-flow-gen-A} is the fact that the matrices $Q_x(t)$ are also well defined and of size $|S_1(x)|$ for degenerate vertices $x \in V$. 
	Moreover, \eqref{eq:curv-flow-gen-Q} provides formulas for the derivatives of the transition probabilities directly and not for the derivatives of their square roots. Using the explicit formulas for the components of $4Q_x(t){\bf 1}_m$ in Lemma \ref{lem:4Q}, the individual equations of \eqref{eq:curv-flow-gen-Q} for all $x \in V$ with $S_1(x) = \{ y_1,\dots,y_m \}$ and all $i \in \{ 1,\dots,m\}$ are then given by
	\begin{equation} \label{eq:curv-flow-gen-p} 
		p_{xy_i}'(t) = p_{xy_i}(t) \left( 2 C_x(t) - p_{y_iy_i}(t) + p_{xx}(t) - 4p_{y_ix}(t) - 2 \sum_{j \neq i} p_{y_iy_j}(t) \right) + \sum_{j \neq i} p_{xy_j}(t) p_{y_jy_i}(t). 
	\end{equation} 
	
	\medskip
	
	In our normalized curvature flow \eqref{eq:curv-flow}, we choose $C_x(t)$ to be the upper curvature bound
	$K_{P(t),\infty}^{d_G(x,\cdot)}(x)$,
	which is expressed in \eqref{eq:KN0-alt} in the transition probabilities in the $1$-ball $B_1(x)$: 
	\begin{equation}
 C_x(t) = K_{P(t),\infty}^{d_G(x,\cdot)}(x) = \frac{1}{2D_x(t)} \left( 4 \sum_{y \in S_1(x)} p_{xy}(t)p_{yx}(t) + \sum_{y,y' \in S_1(x)} p_{xy}(t)p_{yy'}(t) \right) - \frac{p_{xx}(t)}{2}. 
   \end{equation}
	In this special case, the equations \eqref{eq:curv-flow-gen-p} take the explicit form
	$p_{xy_i}'(t) = F_{xy_i}(t)$ with
	\begin{multline} \label{eq:Fxy}
		F_{xy_i}(t) = \\
		p_{xy_i}(t) \left( -4p_{y_ix}(t) - 2 \sum_{j \neq i} p_{y_iy_j}(t) + \frac{4}{D_x} \sum_{j=1}^m p_{xy_j}(t)p_{y_jx}(t) + \frac{1}{D_x} \sum_{j,k} p_{xy_j}(t)p_{y_jy_k}(t) - p_{y_iy_i}(t) \right) \\
		+ \sum_{j \neq i} p_{xy_j}(t)p_{y_jy_i}(t).
	\end{multline}
	
	\begin{rmk}
	We note from \eqref{eq:Fxy} that when an edge with endpoints $x,y$ satisfies $p_{xy}(t_0)=0$ for some time $t_0$ and the edge is not contained in a triangle (more precisely, if there is no directed path $x\to y'\to y$), then we have $p_{xy}(t)=0$ for all $t\ge t_0$.
	\end{rmk}
	
	\subsection{Preservance of Markovian property and curvature sharpness of flow limits} \label{subsec:markovandcurvsharplimit-flow}
	
	In this subsection we provide the proofs of Theorem \ref{thm:flow-pres-MP}
	and Proposition \ref{prop:flow-lim-CS}. Let us start with the proof that the curvature flow preserves the Markovian property and is well defined on the time interval $[0,\infty)$.

	
	\begin{proof}[Proof of Theorem \ref{thm:flow-pres-MP}]
	Note that a solution of the curvature flow is unique and well-defined for some interval $[0,T]$ with $T > 0$, by Picard-Lindel\"of Theorem. We show below that the transition rates $p_{xy}$ are Markovian and are therefore bounded within any such interval. This implies that the solution of the flow is well defined and unique on the whole interval $[0,\infty)$ by standard extension arguments for ordinary differential equations. 
	
    First, we show that $D_x(t)=\sum_{y \neq x} p_{xy}(t) = {\bf{1}}_m^\top {\bf{p}}_x(t)$ stays constant under the flow:
	\[
	D_x'(t) = -4\,{\bf{1}}_m^\top Q_x(t) {\bf{1}}_m + 2C_x(t)D_x(t) = 0
	\]
	with $C_x(t) = K_{P(t),\infty}^{d_G(x,\cdot)}(x) = 2\frac {{\bf{1}}_m^\top Q_x(t) {\bf{1}}_m}{D_x(t)}$ due to Lemma \ref{lem:curvsharp_M}(a).
	Next, we will show that $p_{xy}(t) \geq 0$ for all $x \neq y$ and $t>0$. Suppose not. Let $T>0$ such that $p_{uv}(T)=-\delta <0$ for some $u \neq v$ and some $\delta < 1$.
	Then let
	\[
	t_0 := \inf\{t > 0: e^{-11|V|^2 t} p_{xy}(t) \leq -\eps \mbox{ for some } x\neq y\} \in (0,T]
	\]
	with
	\[
	\eps = e^{-11|V|^2 T} \delta.
	\]
	Let $x \neq y$ be chosen such that the infimum is attained. 
	We first notice that $p_{x'y'}(t_0) \geq p_{xy}(t_0) \geq -\delta \geq -1$ for all $x' \neq y'$. Moreover $p_{xy}(t_0) <0$.
   We estimate at $t=t_0$
	\[
	0 \geq \frac{\partial}{\partial_t}\Big\vert_{t=t_0}  \left(e^{-11|V|^2 t} p_{xy} \right) = e^{-11|V|^2 t_0}(p_{xy}'(t_0) - 11|V|^2 p_{xy}(t_0)),
	\]
	giving
	\[
	p_{xy}'(t_0) \leq   11|V|^2 p_{xy}(t_0).
	\]
	Recall that
	\[
	{\bf{p}}_{xy}'(t_0) = - 4Q_x(t_0) {\bf 1}_m(y) + 2 C_x(t_0) {\bf{p}}_{xy}(t_0).
	\]
	
	By Lemma~\ref{lem:4Q}, we have (dropping henceforth the argument $t_0$ for simplicity),
      $$
	   (4Q_x {\bf{1}}_m)_i = p_{xy_i} \left( p_{y_iy_i} - p_{xx} + 4p_{y_ix} + 2 \sum_{j \neq i} p_{y_iy_j} \right) - \sum_{j \neq i} p_{xy_j} p_{y_jy_i},
      $$
	where we choose $y_i =y$.  We can estimate $p_{x'y'}(t_0) \leq |V|$ for all $x' \neq y'$ because $D_{x'}(t_0) =D_{x'}(0) \leq 1$ and as $p_{x'z}(t_0) \geq -1$ for all $x' \neq z$. This implies that we have for all $x' \neq y'$ 
	$$ -1 \le p_{xy} \le p_{x'y'} \le |V| $$
	with $p_{xy} < 0$.
	Since $p_{xy_j},p_{y_jy_i} \in \big[\, p_{xy},\, |V|\,\big]$, we conclude that
	$$ p_{xy_j} p_{y_jy_i} \ge -|p_{xy}| |V| = p_{xy} |V|. $$
	Moreover, we have
	$$ p_{y_iy_i} - p_{xx} + 4p_{y_ix} + 2 \sum_{j \neq i} p_{y_iy_j} \ge 0 - 1 - 4 - 2 |V| \ge -7|V|. $$
	Bringing everything together, we obtain
	$$ (4Q(x) {\bf{1}}_m)_i \leq p_{xy}(-7 |V|) - p_{xy} |V|^2  \leq -8 |V|^2 p_{xy}.
	$$

	
	As $C_x \geq -1$, we obtain
	\[
	p_{xy}'(t_0) \geq  10|V|^2 p_{xy}(t_0).
	\]
	This is a contradiction to $p_{xy}'(t_0) \leq  11|V|^2 p_{xy}(t_0)$ as $p_{xy}(t_0)<0$.
	This finishes the proof that all transitions rates remain non-negative.
	Together with the Markovian preservance, this implies that all transition rates are bounded above by $1$. 
	
	Moreover, if we start with a non-degenerate weighting scheme, i.e., $p_{xy}(0)>0$ for all $(x,y) \in E$, we can show similarly
	that $p_{xy}'(t) \geq  c p_{xy}(t)$ for some constant $c$ and hence $p_{xy}(t) \geq e^{ct}p_{xy}(0)$ stays positive for all $t>0$, showing that non-degenerate weighting schemes stay non-degenerate during the flow.
	This finishes the proof.
\end{proof}

Finally, we will show that limits of the curvature flow are curvature sharp.
	
We first show that limit points of autonomous ordinary differential equations with locally Lipschitz right hand side are fixed points. We believe that this is standard but for the readers' convenience we provide a proof.
	
\begin{lemma}\label{lem:ODEllmitFixpoint}
Let $F:\R^n \to \R^n$ be locally Lipschitz. Suppose $(u_t)_{t\in [0,\infty)} \in \R^n$ satisfies
\[
\partial_t u_t = F \circ u_t
\]
and
\[
\lim_{t\to \infty } u_t = v
\]
Then we have
\[
F\circ v = 0.
\]
\end{lemma}
	
\begin{proof}
Let $v_t$ be a solution to $v_0=v$ and $\partial_t v_t = F\circ v_t$ on $t \in [0,T]$ for some $T>0$. This exists due to Picard-Lindel\"of's Theorem. Let $\eps >0$.
We aim to show $v_T = v$, that is, $v_t$ is constant in time (since $T > 0$ can be chosen as small as we like). We observe that for any norm $\|\cdot\|$ on $\R^n$,
\[
\|v_T - v\| \leq  \|v_T - u_{s+T} \| + \|v - u_{s+T}\|
\]
for all $s>0$.
 
As the solution is continuous in the initial value by the local Lipschitz condition, there exists $\delta>0$ such that $\|v_T - w_T\|<\eps$ whenever $\|v_0-w_0\|<\delta$
and $w_t$ is a solution to $\partial_t w_t = F \circ w_t$.
As $u_s$ converges to $v$, we can choose $s$ such that $\|u_s-v\| < \delta$ and $\|u_{s+T} - v\| < \eps$. Then,
\[
\|v_T - v\| \leq  \|v_T - u_{s+T} \| + \|v - u_{s+T}\| \leq \eps + \eps.
\]
As $T>0$ and $\eps>0$ can be chosen arbitrarily small, this shows $v_t$ is constant in time $t$ implying $F\circ v = 0$. This finishes the proof.
\end{proof}	

With this lemma, we can now prove that limits of our curvature flow are curvature sharp.

\begin{proof}[Proof of Proposition \ref{prop:flow-lim-CS}]
As the right hand side of the curvature flow is locally Lipschitz, we can apply Lemma~\ref{lem:ODEllmitFixpoint}. Let ${\bf{p}}_x^\infty = \lim_{t \to \infty} {\bf{p}}_x(t)$ for all $x \in V$. 
Then,
\[
-4Q_x {\bf{1}}_m + 2K^{d_G(x,\cdot)} {\bf{p}}_x^\infty = 0.
\]
This is equivalent to curvature sharpness of $(G,P^\infty)$ by Theorem~\ref{thm:curvsharpeq}. Thus, the proof is finished.
\end{proof}

{\bf{Acknowledgement:}} Shiping Liu is supported by the National Key R and D Program of China 2020YFA0713100 and the National Natural Science Foundation of China (No. 12031017). We like to thank the London Mathematical Society for their support of Ben Snodgrass via the Undergraduate Research Bursary URB-2021-02, during which the curvature flow was implemented and which lead to many of the research results presented in this paper. David Cushing is supported by the Leverhulme Trust Research
Project Grant number RPG-2021-080.

\bibliographystyle{alpha}	
\bibliography{References}

\newcommand{\etalchar}[1]{$^{#1}$}
\begin{thebibliography}{CKPWM20}

\bibitem[Alb69]{Alb-69}
Arthur Albert.
\newblock Conditions for positive and nonnegative definiteness in terms of
  pseudoinverses.
\newblock {\em SIAM J. Appl. Math.}, 17:434--440, 1969.

\bibitem[BCLL17]{Bauer17curvature}
Frank Bauer, Fan Chung, Yong Lin, and Yuan Liu.
\newblock Curvature aspects of graphs.
\newblock {\em Proceedings of the American Mathematical Society},
  145(5):2033--2042, 2017.

\bibitem[BE84]{BE84}
Dominique Bakry and Michel \'{E}mery.
\newblock Hypercontractivit\'{e} de semi-groupes de diffusion.
\newblock {\em C. R. Acad. Sci. Paris S\'{e}r. I Math.}, 299(15):775--778,
  1984.

\bibitem[BLL{\etalchar{+}}20]{Bai21OllRicflow}
Shuliang Bai, Yong Lin, Linyuan Lu, Zhiyu Wang, and Shing-Tung Yau.
\newblock Ollivier {R}icci-flow on weighted graphs.
\newblock {\em arXiv preprint arXiv:2010.01802}, 2020.

\bibitem[CKK{\etalchar{+}}21]{CKKLP21}
David Cushing, Supanat Kamtue, Riikka Kangaslampi, Shiping Liu, and Norbert
  Peyerimhoff.
\newblock Curvatures, graph products and {R}icci flatness.
\newblock {\em J. Graph Theory}, 96(4):522--553, 2021.

\bibitem[CKL{\etalchar{+}}22a]{CKLMPS-22}
David Cushing, Supanat Kamtue, Shiping Liu, Florentin Münch, Norbert
  Peyerimhoff, and Ben Snodgrass.
\newblock Bakry-{{\'E}}mery curvature sharpness and curvature flow in finite
  weighted graphs. {II}. {I}mplementation.
\newblock {\em arXiv preprint arXiv:2212.12401}, 2022.

\bibitem[CKL{\etalchar{+}}22b]{CKLLS19}
David Cushing, Riikka Kangaslampi, Valtteri Lipi{\"a}inen, Shiping Liu, and
  George~W Stagg.
\newblock The graph curvature calculator and the curvatures of cubic graphs.
\newblock {\em Experimental Mathematics}, 31(2):583--595, 2022.

\bibitem[CKLP22]{CKLP-21}
David Cushing, Supanat Kamtue, Shiping Liu, and Norbert Peyerimhoff.
\newblock Bakry-\'{E}mery curvature on graphs as an eigenvalue problem.
\newblock {\em Calc. Var. Partial Differential Equations}, 61(2):Paper No. 62,
  33, 2022.

\bibitem[CKPWM20]{CKPW20}
David Cushing, Supanat Kamtue, Norbert Peyerimhoff, and Leyna Watson~May.
\newblock Quartic graphs which are {B}akry-\'{E}mery curvature sharp.
\newblock {\em Discrete Math.}, 343(3):111767, 15, 2020.

\bibitem[CL03]{CL-03}
Bennett Chow and Feng Luo.
\newblock Combinatorial {R}icci flows on surfaces.
\newblock {\em J. Differential Geom.}, 63(1):97--129, 2003.

\bibitem[CLP20]{CLP-20}
David Cushing, Shiping Liu, and Norbert Peyerimhoff.
\newblock Bakry-\'{E}mery curvature functions on graphs.
\newblock {\em Canad. J. Math.}, 72(1):89--143, 2020.

\bibitem[CLY14]{CLY14}
Fan Chung, Yong Lin, and S.-T. Yau.
\newblock Harnack inequalities for graphs with non-negative {R}icci curvature.
\newblock {\em J. Math. Anal. Appl.}, 415(1):25--32, 2014.

\bibitem[DL22]{DL-22}
Karel Devriendt and Renaud Lambiotte.
\newblock Discrete curvature on graphs from the effective resistance.
\newblock {\em arXiv preprint arXiv:2201.06385}, 2022.

\bibitem[EK20]{EK20superRicflow}
Matthias Erbar and Eva Kopfer.
\newblock Super {R}icci flows for weighted graphs.
\newblock {\em J. Funct. Anal.}, 279(6):108607, 51, 2020.

\bibitem[Elw91]{Elworthy91}
K.~D. Elworthy.
\newblock Manifolds and graphs with mostly positive curvatures.
\newblock In {\em Stochastic analysis and applications ({L}isbon, 1989)},
  volume~26 of {\em Progr. Probab.}, pages 96--110. Birkh\"{a}user Boston,
  Boston, MA, 1991.

\bibitem[EM12]{EM12}
Matthias Erbar and Jan Maas.
\newblock Ricci curvature of finite {M}arkov chains via convexity of the
  entropy.
\newblock {\em Arch. Ration. Mech. Anal.}, 206(3):997--1038, 2012.

\bibitem[FL22]{FL22WarpedProd}
Zohreh Fathi and Sajjad Lakzian.
\newblock Bakry-\'{E}mery {R}icci curvature bounds for doubly warped products
  of weighted spaces.
\newblock {\em J. Geom. Anal.}, 32(3):Paper No. 79, 75, 2022.

\bibitem[For03]{Forman}
Robin Forman.
\newblock Bochner's method for cell complexes and combinatorial {R}icci
  curvature.
\newblock {\em Discrete Comput. Geom.}, 29(3):323--374, 2003.

\bibitem[FS18]{FS18}
Max Fathi and Yan Shu.
\newblock Curvature and transport inequalities for {M}arkov chains in discrete
  spaces.
\newblock {\em Bernoulli}, 24(1):672--698, 2018.

\bibitem[HL19]{HL19}
Bobo Hua and Yong Lin.
\newblock Graphs with large girth and nonnegative curvature dimension
  condition.
\newblock {\em Comm. Anal. Geom.}, 27(3):619--638, 2019.

\bibitem[JL14]{JostLiu2014}
J\"{u}rgen Jost and Shiping Liu.
\newblock Ollivier's {R}icci curvature, local clustering and
  curvature-dimension inequalities on graphs.
\newblock {\em Discrete Comput. Geom.}, 51(2):300--322, 2014.

\bibitem[JM21]{JM21}
J{\"u}rgen Jost and Florentin M{\"u}nch.
\newblock Characterizations of {F}orman curvature.
\newblock {\em arXiv preprint arXiv:2110.04554}, 2021.

\bibitem[KKRT16]{KKRT16}
Bo'az Klartag, Gady Kozma, Peter Ralli, and Prasad Tetali.
\newblock Discrete curvature and abelian groups.
\newblock {\em Canad. J. Math.}, 68(3):655--674, 2016.

\bibitem[KMY21]{KMY21}
Mark Kempton, Florentin M\"{u}nch, and Shing-Tung Yau.
\newblock A homology vanishing theorem for graphs with positive curvature.
\newblock {\em Comm. Anal. Geom.}, 29(6):1449--1473, 2021.

\bibitem[LMP]{LMP17rigidity}
Shiping Liu, Florentin M{\"u}nch, and Norbert Peyerimhoff.
\newblock Rigidity properties of the hypercube via {B}akry-{\'e}mery curvature.
\newblock {\em to appear in Math. Ann.}

\bibitem[LMP18]{LMP18}
Shiping Liu, Florentin M\"{u}nch, and Norbert Peyerimhoff.
\newblock Bakry-\'{E}mery curvature and diameter bounds on graphs.
\newblock {\em Calc. Var. Partial Differential Equations}, 57(2):Paper No. 67,
  9, 2018.

\bibitem[LMPR19]{LMPR19}
Shiping Liu, Florentin M\"{u}nch, Norbert Peyerimhoff, and Christian Rose.
\newblock Distance bounds for graphs with some negative {B}akry-\'{E}mery
  curvature.
\newblock {\em Anal. Geom. Metr. Spaces}, 7(1):1--14, 2019.

\bibitem[LP18]{LP18}
Shiping Liu and Norbert Peyerimhoff.
\newblock Eigenvalue ratios of non-negatively curved graphs.
\newblock {\em Combin. Probab. Comput.}, 27(5):829--850, 2018.

\bibitem[LY10]{LinYau2010}
Yong Lin and Shing-Tung Yau.
\newblock Ricci curvature and eigenvalue estimate on locally finite graphs.
\newblock {\em Math. Res. Lett.}, 17(2):343--356, 2010.

\bibitem[Ma13]{Ma13Bochner}
Li~Ma.
\newblock {B}ochner formula and {B}ernstein type estimates on locally finite
  graphs.
\newblock {\em arXiv preprint arXiv:1304.0290}, 2013.

\bibitem[Man15]{Man15LogHarnark}
Shoudong Man.
\newblock Logarithmic {H}arnack inequalities for general graphs with positive
  {R}icci curvature.
\newblock {\em Differential Geom. Appl.}, 38:33--40, 2015.

\bibitem[MR20]{MR20}
Florentin M\"{u}nch and Christian Rose.
\newblock Spectrally positive {B}akry-\'{E}mery {R}icci curvature on graphs.
\newblock {\em J. Math. Pures Appl. (9)}, 143:334--344, 2020.

\bibitem[M{\"u}n18]{Munch18Perpetual}
Florentin M{\"u}nch.
\newblock Perpetual cutoff method and discrete {R}icci curvature bounds with
  exceptions.
\newblock {\em arXiv preprint arXiv:1812.02593}, 2018.

\bibitem[M{\"u}n19]{Munch19LiYau}
Florentin M{\"u}nch.
\newblock Li-{Y}au inequality under ${CD}(0,n)$ on graphs.
\newblock {\em arXiv preprint arXiv:1909.10242}, 2019.

\bibitem[NLLG19]{Ni19commDetect}
Chien-Chun Ni, Yu-Yao Lin, Feng Luo, and Jie Gao.
\newblock Community detection on networks with {R}icci flow.
\newblock {\em Scientific reports}, 9(1):1--12, 2019.

\bibitem[Oll09]{Ollivier}
Yann Ollivier.
\newblock Ricci curvature of {M}arkov chains on metric spaces.
\newblock {\em J. Funct. Anal.}, 256(3):810--864, 2009.

\bibitem[Oll10]{OllProbs}
Yann Ollivier.
\newblock A survey of {R}icci curvature for metric spaces and {M}arkov chains.
\newblock In {\em Probabilistic approach to geometry}, volume~57 of {\em Adv.
  Stud. Pure Math.}, pages 343--381. Math. Soc. Japan, Tokyo, 2010.

\bibitem[PES{\etalchar{+}}16]{PESTGT16}
Maryam Pouryahya, Rena Elkin, Romeil Sandhu, Sarah Tannenbaum, Tryphon
  Georgiou, and Allen Tannenbaum.
\newblock Bakry-{{\'E}}mery {R}icci curvature on weighted graphs with
  applications to biological networks.
\newblock In {\em Int. Symp. on Math. Theory of Net. and Sys}, volume~22,
  page~52, 2016.

\bibitem[Rob19]{Robertson19Harnack}
Sawyer~Jack Robertson.
\newblock Harnack inequality for magnetic graphs.
\newblock {\em arXiv preprint arXiv:1910.04019}, 2019.

\bibitem[Sal21a]{Salez21cutoff}
Justin Salez.
\newblock Cutoff for non-negatively curved {M}arkov chains.
\newblock {\em arXiv preprint arXiv:2102.05597}, 2021.

\bibitem[Sal21b]{Salez21sparse}
Justin Salez.
\newblock Sparse expanders have negative curvature.
\newblock {\em arXiv preprint arXiv:2101.08242}, 2021.

\bibitem[Sch99]{Schmuck}
Michael Schmuckenschl\"{a}ger.
\newblock Curvature of nonlocal {M}arkov generators.
\newblock In {\em Convex geometric analysis ({B}erkeley, {CA}, 1996)},
  volume~34 of {\em Math. Sci. Res. Inst. Publ.}, pages 189--197. Cambridge
  Univ. Press, Cambridge, 1999.

\bibitem[Sic21]{Sic21}
Viola Siconolfi.
\newblock Ricci curvature, graphs and eigenvalues.
\newblock {\em Linear Algebra Appl.}, 620:242--267, 2021.

\bibitem[SY20]{ShiYu20}
Yongjie Shi and Chengjie Yu.
\newblock Comparisons of {D}irichlet, {N}eumann and {L}aplacian eigenvalues on
  graphs and applications.
\newblock {\em arXiv preprint arXiv:2011.04160}, 2020.

\bibitem[TDGC{\etalchar{+}}21]{TDGCDB-22}
Jake Topping, Francesco Di~Giovanni, Benjamin~Paul Chamberlain, Xiaowen Dong,
  and Michael~M Bronstein.
\newblock Understanding over-squashing and bottlenecks on graphs via curvature.
\newblock {\em arXiv preprint arXiv:2111.14522}, 2021.

\bibitem[WSJ17]{WSJ17FormanRicflow}
Melanie Weber, Emil Saucan, and J\"{u}rgen Jost.
\newblock Characterizing complex networks with {F}orman-{R}icci curvature and
  associated geometric flows.
\newblock {\em J. Complex Netw.}, 5(4):527--550, 2017.

\end{thebibliography}

\end{document}